\documentclass[11pt,reqno]{amsart}
\usepackage{amscd,amsmath,amsxtra,amsthm,amssymb,stmaryrd,xr,mathrsfs,mathtools,enumerate,commath, comment, enumitem}
\usepackage{stmaryrd}
\usepackage{xcolor}
\usepackage{commath}
\usepackage{comment}
\usepackage{tikz-cd}
\usepackage[top=25mm,bottom=25mm,left=25mm,right=25mm]{geometry}
\usepackage{longtable} % for 'longtable' environment
\usepackage{pdflscape} % for 'landscape' environment
\usepackage{booktabs}
\usepackage{hyperref}
\definecolor{vegasgold}{rgb}{0.77, 0.7, 0.35}
\definecolor{darkgoldenrod}{rgb}{0.72, 0.53, 0.04}
\definecolor{gold(metallic)}{rgb}{0.83, 0.69, 0.22}
\hypersetup{
 colorlinks=true,
 linkcolor=darkgoldenrod,
 filecolor=brown,      
 urlcolor=gold(metallic),
 citecolor=darkgoldenrod,
 }

\usepackage[all,cmtip]{xy}

\DeclareFontFamily{U}{wncy}{}
\DeclareFontShape{U}{wncy}{m}{n}{<->wncyr10}{}
\DeclareSymbolFont{mcy}{U}{wncy}{m}{n}
\DeclareMathSymbol{\Sh}{\mathord}{mcy}{"58}
\usepackage[T2A,T1]{fontenc}
\usepackage[OT2,T1]{fontenc}

\newtheorem{theorem}{Theorem}[section]
\newtheorem{lemma}[theorem]{Lemma}

\newtheorem*{theorem*}{Theorem}
\newtheorem*{ass*}{Assumption}
\newtheorem{definition}[theorem]{Definition}
\newtheorem{corollary}[theorem]{Corollary}
\newtheorem{remark}[theorem]{Remark}

\newtheorem{proposition}[theorem]{Proposition}

\newcommand{\Z}{\mathbb{Z}}

\newcommand{\Q}{\mathbb{Q}}
\newcommand{\F}{\mathbb{F}}

\newcommand{\cL}{\mathcal{L}}

\newcommand{\G}{\mathcal{G}}

\newcommand{\op}[1]{\operatorname{#1}}

\newcommand\mtx[4] { \left( {\begin{array}{cc}
 #1 & #2 \\
 #3 & #4 \\
 \end{array} } \right)}

\numberwithin{equation}{section}

\begin{document}

\title[Rank stability in metabelian extensions]{Rank stability of elliptic curves
in \\ certain non-abelian extensions}

\author[S.~Pathak]{Siddhi Pathak}
\address[Pathak]{Chennai Mathematical Institute, H1, SIPCOT IT Park, Kelambakkam, Siruseri, Tamil Nadu 603103, India}
\email{siddhi@cmi.ac.in}

\author[A.~Ray]{Anwesh Ray}
\address[Ray]{Chennai Mathematical Institute, H1, SIPCOT IT Park, Kelambakkam, Siruseri, Tamil Nadu 603103, India}
\email{anwesh@cmi.ac.in}

\keywords{Rank stability, Malle-Bhargava principal, Selmer groups of elliptic curves}
\subjclass[2020]{11G05, 11R45}
\thanks{Research of the first author was partially supported by an INSPIRE fellowship.}

\begin{abstract}
Let $E_{/\Q}$ be an elliptic curve with rank $E(\Q)=0$. Fix an odd prime $p$, a positive integer $n$ and a finite abelian extension $K/\Q$ with rank $E(K) = 0$. In this paper, we show that there exist infinitely many extensions $L/K$ such that $L/\Q$ is Galois with $\operatorname{Gal}(L/\Q) \simeq \operatorname{Gal}(K/\Q) \ltimes \mathbb{Z}/p^n\mathbb{Z}$, and rank $E(L)=0$. This is an extension of earlier results on rank stability of elliptic curves in cyclic extensions of prime power order to a non-abelian setting. We also obtain an asymptotic lower bound for the number of such extensions, ordered by their absolute discriminant.

% Let $E_{/\mathbb{Q}}$ be an elliptic curve and $K/\mathbb{Q}$ be a finite abelian extension with ${\text{rank } E(K)=0}$. Fix an odd prime $p$ and a positive integer $n$. In this paper, we show that there exist infinitely many Galois extensions $L/K/\Q$ such that ${\operatorname{Gal}(L/\Q) \simeq \operatorname{Gal}(K/\Q) \ltimes \mathbb{Z}/p^n\mathbb{Z}}$ and rank $E(L)$ remains zero. We also This extends earlier results on rank stability in cyclic extensions of prime power order.

% Fix an odd prime number $p$, a finite abelian group $B$ and an elliptic curve $E/\mathbb{Q}$. Let $K/\Q$ be a Galois extension with $\operatorname{Gal}(K/\Q) \simeq B$. Suppose that $E(K)$ has rank $0$. Given a positive integer $n$, let $T \simeq \mathbb{Z}/p^n\mathbb{Z}$ and set $\G = B \ltimes T$. In this paper, we show that there exist infinitely many extensions $L/K/\Q$ such that $\operatorname{Gal}(L/\Q) \simeq \G$ and the rank of $E(L)$ remains $0$. We also prove asymptotic lower bound for the number of such extensions, ordered by their absolute discriminant. This extends results on rank stability in cyclic extensions of prime power order. 
\end{abstract}

\maketitle

\section{Introduction}
\subsection{The setting and historical background} 

\bigskip 

\par An elliptic curve is a smooth, projective, algebraic curve of genus $1$, equipped with a distinguished point at $\infty$. These curves play a central role in number theory, algebraic geometry, and cryptography. The \emph{Mordell–Weil group} of an elliptic curve $E$, defined over a number field $K$, is the group of $K$-rational points $E(K)$, under a natural group law arising from the geometry of the curve. By the \emph{Mordell–Weil theorem}, $E(K)$ is a finitely generated abelian group, and can be expressed as  
\[
E(K) \cong \mathbb{Z}^r \oplus T,
\]
where:
\begin{itemize}
    \item $r \geq 0$ is an integer known as the \textit{Mordell–Weil rank} of $E$ over $K$, which measures the number of independent $K$-rational points of infinite order,
    \item $T$ is a finite abelian group, called the \textit{torsion subgroup}, consisting of points with finite order.
\end{itemize}
\par The rank $r$ is a deep arithmetic invariant of the elliptic curve and is the subject of intense study. The famous Birch and Swinnerton–Dyer conjecture provides a profound connection between the rank $r$ and the behavior of the Hasse–Weil $L$-function $L(E_{/K}, s)$ of $E$ at $s = 1$. Specifically, the conjecture asserts that $\operatorname{rank } E(K)$ is equal to the order of vanishing of $L(E_{/K}, s)$ at $s = 1$.

\par The points on $E(\overline{K})$, where $\overline{K}$ is the algebraic closure of $K$, naturally come equipped with an action of the absolute Galois group $\operatorname{Gal}(\overline{K}/K)$. This interaction allows for the arithmetic structure of $E(\overline{K})$ to be analyzed through Galois theoretic methods. A key tool in this analysis is the theory of Selmer groups, which consist of global Galois cohomology classes satisfying local conditions at all places. Selmer groups also carry information about the mysterious Tate–Shafarevich group $\Sh(E/K)$, whose elements consist of homogenous spaces measuring the failure of certain local-global principles for rational points on $E$.

\par Understanding the behaviour of the rank of an elliptic curve over extensions of number fields is a classical problem in number theory. One aspect of this study is to recognize the variation in the rank of an elliptic curve over field extensions. In particular, given an elliptic curve $E_{/\Q}$, one is interested in knowing how often is $\op{rank} E(L)>\op{rank} E(\Q)$, as $L/\Q$ varies over a family of number fields with fixed degree, ordered by absolute discriminant. In this regard, David, Fearnley and Kisilevsky \cite{DFK1, DFK2} studied the phenomenon of rank jumps of elliptic curves in cyclic extensions of prime order. More specifically, given an elliptic curve $E_{/\Q}$ and a prime $p$, they considered all cyclic extensions $L/\Q$ with Galois group isomorphic to $\Z/p\Z$, and asked how often $\op{rank} E(L)>\op{rank} E(\Q)$. Using the relationship between modular symbols and special values of cyclic twists of the Hasse-Weil $L$-function of $E$, they made predictions based on random matrix theory. In particular, it is conjectured that if $p\geq 7$, then there are only finitely many $\Z/p\Z$-extensions in which the rank can increase. These assertions are far from being resolved. \\

Mazur and Rubin \cite{MazurRubinQtwists} study distribution questions for $2$-primary Selmer groups in quadratic twist families of elliptic curves. Their methods have applications to showing that the rank remains stable for a large number of quadratic twists. Their results imply (under additional conditions) that the rank can also become arbitrarily large after base change by a quadratic extension. For any prime $p$, the growth and stability for $p$-Selmer groups of elliptic curves in $p$-cyclic extensions of a number field has been studied by Klagsbrun, Mazur and Rubin \cite{KMR}. Similar results have been obtained by Smith, see \cite{smith2022distribution1,smith2022distribution2} for further details. Park \cite{Park} has studied analogous questions over function fields. The phenomenon of rank jumps has been investigated further in other families of number field extensions. See for instance
Lemke-Oliver--Thorne \cite{Oliverthorne}, Schnidman--Weiss \cite{ShnidmanWeiss}, Beneish--Kundu--Ray \cite{BKR}, Berg--Ryan--Young \cite{Berg}. Moreover, the study of such rank stability questions has close connection to Hilbert's tenth problem, which has been recently studied for certain families of field extensions by Garcia--Fritz and Pasten \cite{garcia-fritzpasten} and Kundu--Lei--Sprung \cite{KLS}.\\ 

One advancement in this direction follows from the study of diophantine stability by Mazur-Rubin \cite{mazur2018diophantine}. They establish that for a simple abelian variety $A_{/K}$ with $\op{End}_{\overline{K}}(A)=\op{End}_K(A)$, there is a set $S$ of primes with positive density such that for all $\ell \in S$ and for all integers $n>0$, there are infinitely many $ \left(\Z/\ell^n \Z \right)$-extensions $L/K$ that have $A(L)=A(K)$. This phenomenon has been studied by Ray--Weston from a different perspective in \cite{ray2023diophantine}, where it is shown that for any prime $\ell\geq 5$, and any number field $K$, the density of elliptic curves $E_{/\Q}$, that are diophantine stable over $K$ at $\ell$ is $1$. \\
% Such investigations are, in part, motivated by applications to the Hilbert's tenth problem for number rings.
% That is, most elliptic curves are diophantine stable at any given prime $\ell$.\\

Another question of interest in number theory is the study of Galois extensions of a fixed number field with a given Galois group. Conjectures concerning the number of such extensions, ordered by their absolute discriminant, were proposed by Malle \cite{Malle1, Malle2}. These can be viewed as statistical generalizations of the inverse Galois problem for number fields. There has been considerable progress towards these predictions, especially when the Galois group is of a specific kind. For instance, Malle's conjecture has been proved for abelian groups by M\"aki \cite{maki1985density} and Wright \cite{wright1989distribution} and for certain nilpotent groups by Koymans and Pagano \cite{koymans2023malle}. Although the inverse Galois problem for finite solvable groups over number fields was resolved by Shafarevich \cite{Shaferevich}, Malle's conjecture for general solvable groups is still open.\\

In this paper, we merge the above two themes to address the problem of counting certain solvable extensions of $\Q$ that have the rank stability property. More specifically, given a finite abelian group $B$, an elliptic curve $E_{/\mathbb{Q}}$ with rank $E(\Q) = 0$, a rational prime $\ell$ and a positive integer $n>0$, under mild hypothesis, we demonstrate the existence of infinitely many extensions ${L}/\mathbb{Q}$ with $\operatorname{Gal}({L}/\mathbb{Q}) \simeq B \ltimes \Z/\ell^n \Z$ such that rank $E(L) =$ rank $E(K) =$ rank $E(\Q) = 0$. Additionally, we obtain a lower bound for the number of such metabelian extensions, ordered by their absolute discriminant. This can compared with the total number of $B \ltimes \Z/\ell^n \Z$ extensions of $\Q$ predicted by Malle's conjectures. \\

\subsection{Statement of the main result}
Let $p$ be an odd prime number and $n$ be a positive integer. Let $T$ denote the cyclic group $\Z/p^n \Z$. Let $B$ be an abelian group and 
$$\chi_0:B\rightarrow \op{Aut}(T)\xrightarrow{\sim}(\Z/p^n \Z)^\times$$ 
be a character. Let $\G:=B\ltimes T$ with respect to $\chi_0$. Here, the semi-direct product is defined by the relation $x h x^{-1}=\chi_0(x) h$ for all $x\in B$ and $h\in T$. Fix a Galois extension $K/\Q$ such that $\op{Gal}(K/\Q)$ is isomorphic to $B$. \\ 

A $(\G,K)$-extension $L/K/\Q$ is a Galois extension $L/\Q$ which contains $K$ and there exists an isomorphism $\op{Gal}(L/\Q)\xrightarrow{\sim} \G$ which induces an isomorphism of $\op{Gal}(L/K)$ with $T$. Thus, identifying $\op{Gal}(L/\Q)$ with $\G$, we find that $L^T=K$. Throughout the paper, we denote the absolute discriminant of $L$ as $\Delta_L$. \\

Denote by $\op{Sel}_p(E/L)$ the $p$-Selmer group of $E$ over $L$. Then, there is a natural short exact sequence
\[0\rightarrow E(L)/p E(L)\rightarrow \op{Sel}_p(E/L)\rightarrow \Sh(E/L)[p]\rightarrow 0.\]
Thus, if $\op{Sel}_p(E/L)=0$, then, in particular, the rank of $E(L)$ and the $p$-primary part of the Tate-Shafarevich group $\Sh(E/L)$ are zero.  \\

With the above notation, our main theorem is as follows.
\begin{theorem}\label{main thm intro}
    Let $p\geq 5$ be a prime number, $K/\Q$ be an abelian number field and $E_{/\Q}$ be an elliptic curve. Assume that the following conditions are satisfied.
\begin{enumerate}
    \item\label{c1 of main thm} The $p$-Selmer group $\op{Sel}_p(E/K)$ vanishes.
    \item\label{c2 of main thm} The field $K$ contains $\Q(\mu_{p^n})$. In particular, $B$ has order divisible by $p^{n-1}(p-1)$.
    \item\label{c3 of main thm} The character $\chi_0$ is nontrivial modulo $p$.
    \item\label{c4 of main thm} The intersection $K\cap \Q(E[p])$ is equal to $\Q(\mu_p)$.
    \item\label{c5 of main thm} The Galois representation 
    $\rho_{E, p}:\op{G}_{\Q}\rightarrow \op{GL}_2(\F_p)$ associated with the $p$-torsion of $E(\overline{\Q})$ is surjective.
\end{enumerate}
Then there are infinitely many $(\G,K)$-extensions $L/K/\Q$ such that $\op{Sel}_p(E/L)=0$. In particular, the rank of $E(L)$ is zero. Moreover, if $N_{\mathcal{G},K}(E;x)$ denotes the total number of $(\mathcal{G}, K)$-extensions such that $\op{Sel}_p(E/L)=0$ and $|\Delta_L|\leq x$, then we have
\[N_{\mathcal{G},K}(E;x)\gg_{\G,K}  \, \, x^{\frac{1}{|B|(p^n-1)}} \, \, (\log x)^{\left(p^{n-1}(p-1) \alpha\right)-1},\] where 
\[\alpha:=\left(\frac{p^2-p-1}{[K:\Q(\mu_p)](p-1)(p^2-1)}\right).\]
\end{theorem}
The above result is obtained as an immediate consequence of Theorem \ref{main thm}, proved in Section 4 and Theorem \ref{discriminant density theorem}, proved in Section 5.\\

Given a positive real number $x>0$, set $M_{\G, K}(x)$ to denote the total number of $(\mathcal{G}, K)$-extensions $L/K/\Q$ such that $|\Delta_L|\leq x$. A variant of the \emph{Malle-Bhargava principle} predicts that 
\begin{equation}
    M_{\G, K}(x) \, \sim \, c \, x^{\frac{1}{|B| \, (p-1) \, p^{n-1}}}  \, (\log x)^b,
\end{equation}
where $c>0$ and $b\in \Z$ are constants that depend on $\G$ and $K$. For further details, see \cite[p. 2519 and p.2530]{Alberts} and \cite{BhargavaIMRN}. Thus, the number $N_{\mathcal{G},K}(E;x)$ differs from the total number $M_{\G, K}(x)$ up to a power of $x$ to the order of $O\left(\frac{1}{|B|p^{n+1}}\right)$ and a factor of $\log x$.\\

An interesting class of examples to which Theorem \ref{main thm intro} applies are when $K=\Q(\mu_{p^n})$ and $\chi_0(x)=x^k$, where $k$ is not divisible by $(p-1)$. In this case, $\chi_0$ is $k$-th power of the mod-$p^n$ cyclotomic character. Then the conditions \eqref{c2 of main thm} and \eqref{c3 of main thm} of Theorem \ref{main thm intro} are clearly satisfied. If the representation $\rho_{E,p}$ is surjective, then \[\op{Gal}(\Q(E[p])/\Q(\mu_p))\simeq \op{SL}_2(\F_p).\] Since $p\geq 5$, the group $\op{PSL}_2(\F_p)$ is simple, $\op{SL}_2(\F_p)$ does not have any abelian quotients, with cardinality divisible by $p$. Hence, the condition \eqref{c4 of main thm} follows from \eqref{c5 of main thm} when $K=\Q(\mu_{p^n})$, and Theorem \ref{main thm intro} is applicable. It is worth noting that the representation $\rho_{E,p}$ is surjective for most elliptic curves, when ordered according to height by a theorem of Duke \cite{duke1997elliptic}.\\

\par When $n=1$, so $K = \Q(\mu_p)$, we obtain a stronger result (Theorem \ref{n=1 thm} from Section \ref{s 3}) which is also valid for $p=3$ and easier to state. At the end of the Section \ref{s 3}, we provide an explicit example for which our hypotheses are satisfied. 

\subsection{Method of proof} We now describe the ideas that go into the proofs of Theorem \ref{main thm} and Theorem \ref{discriminant density theorem}. In Section \ref{s 2}, given a $p^n$ cyclic extension $L/K$ and an elliptic curve $E_{/K}$ such that $\op{Sel}_p(E/K)=0$, we derive conditions for $\op{Sel}_p(E/L)=0$. Proposition \ref{cor 2.7} asserts that if the following conditions hold, then, $\op{Sel}_p(E/L)=0$.
\begin{enumerate}
        \item The Selmer group $\op{Sel}_p(E/K)$ is $0$.
        \item All primes of $K$ that lie above $p$ and the primes of bad reduction for $E$ are completely split in $L$.
        \item All primes $v$ of $K$ that ramify satisfy $\widetilde{E}(k_v)[p]=0$. Here, $\widetilde{E}$ denotes the reduction of $E$ at $v$ and $k_v$ is the residue field at $v$.
    \end{enumerate}
Suppose that $E$ and $K$ satisfy the conditions of Theorem \ref{main thm intro}. In Section \ref{s 4}, we parameterize $(\G,K)$-extensions satisfying certain local conditions by Selmer classes. These are global Galois $1$-cohomology classes taking values in $\Z/p^n \Z(\chi)$ and satisfying various local conditions at a certain set of primes $Y$. We show that these extensions satisfy the hypotheses of Proposition \ref{cor 2.7} mentioned above. Finally, we count the number of such cohomology classes. In order to do this, we invoke a formula due to Wiles (cf. Theorem \ref{wiles thm}) that relates the size of the Selmer group to that of a certain dual Selmer group, along with local terms arising from the Selmer conditions. The set of primes $Y$ will be chosen to contain a prescribed set of primes $Z$ so that the dual Selmer group becomes independent of $Y$. These choices play a crucial role in obtaining the main results Theorem \ref{main thm}, that is, when $n\geq 1$. The density of these extensions ordered by discriminant is proved in Theorem \ref{discriminant density theorem}, which relies on Theorem \ref{main thm} and an application of Delange's Tauberian theorem. This result is proved in Section \ref{s 5}.

\bigskip

%\subsection*{Conflict of interest statement} There are no conflicts of interest to report.
%\subsection*{Data availability statement} No data was analyzed in this article.

\section{\bf Growth of $p$-Selmer groups in $p$-extensions}\label{s 2}

\par Let $p$ be an odd prime number and $K$ be a number field. Given an elliptic curve $E_{/K}$, and an algebraic extension $L/K$, we denote by $E_{/L}$ the base changed curve over $L$ and $E(L)$ its Mordell-Weil group. Fix an algebraic closure $\bar{L}$ of $L$ and set $\op{G}_L$ to denote the absolute Galois group $\op{Gal}(\bar{L}/L)$. Let $\Omega_L$ be the set of all non-archimedean places of $L$ and for $v\in \Omega_L$, denote by $L_v$ the completion of $L$ at $v$. Choose an algebraic closure $\bar{L}_v$ and an embedding $\iota_v: \bar{L}\hookrightarrow \bar{L}_v$. Setting $\op{G}_{L_v}:=\op{Gal}(\bar{L}_v/L_v)$, we find that this induces an embedding $\iota_v^*: \op{G}_{L_v}\hookrightarrow \op{G}_L$. For ease of notation, we shall use \[\begin{split} & H^i(L, \cdot):=H^i(\op{G}_{L}, \cdot),\\ 
& H^i(L_v, \cdot):=H^i(\op{G}_{L_v}, \cdot)\end{split}\] for $i\geq 0$ and $v\in \Omega_L$. Given a global cohomology class $f$, denote by $\op{res}_v(f)$ the restriction of $f$ to $\op{G}_{L_v}$.\\

\par For $n\geq 1$ denote by $E[p^n]$ the $p^n$-torsion subgroup of $E(\bar{K})$; set $E[p^\infty]:=\bigcup_{n\geq 1} E[p^n]$. Let $\mathfrak{G}$ denote either $\op{G}_L$ or $\op{G}_{L_v}$ for some $v\in \Omega_L$. Consider the short exact sequence of $\mathfrak{G}$-modules 
\[0\rightarrow E[p^n]\rightarrow E\xrightarrow{\times p^n} E\rightarrow 0.\] We obtain a long exact sequence in cohomology from which we derive the following exact sequence
\[0\rightarrow E(\cL)\otimes_{\Z} \Z/p^n \Z\xrightarrow{\kappa_{\cL,n}} H^1(\cL, E[p^n])\rightarrow H^1(\cL, E)[p^n]\rightarrow 0,\]
where $\cL$ denotes $L$ or $L_v$. Passing to the direct limit, we obtain an important short exact sequence
\[0\rightarrow E(\cL)\otimes_{\Z} \Q_p/\Z_p \xrightarrow{\kappa_{\cL,\infty}} H^1(\cL, E[p^\infty])\rightarrow H^1(\cL, E)[p^\infty]\rightarrow 0.\]

\par The Selmer group is then defined as 
\[\op{Sel}_{p^n}(E/L):=\op{ker}\left(H^1(L, E[p^n])\longrightarrow \prod_{v\in \Omega_L} H^1(L_v, E)[p^n]\right).\] In other words, $\op{Sel}_{p^n}(E/L)$ consists of classes $f\in H^1(L, E[p^n])$ such that for all $v\in \Omega_L$, we find that $\op{res}_v (f)\in \op{Image}(\kappa_{L_v, n})$. Let $\Sh(E/L)$ denote the Tate-Shafarevich group of $E_{/L}$. The Selmer group fits into a short exact sequence
\begin{equation}\label{Selmer TS sequence}0\rightarrow E(L)\otimes \Z/p^n \rightarrow \op{Sel}_{p^n}(E/L)\rightarrow \Sh(E/L)[p^n]\rightarrow 0. \end{equation} Taking the direct limit over $n$, we have the following short exact sequence
\[0\rightarrow E(L)\otimes \Q_p/\Z_p\rightarrow \op{Sel}_{p^\infty}(E/L)\rightarrow \Sh(E/L)[p^\infty]\rightarrow 0.\]

\begin{definition}\label{choice of S} Let $S$ be the finite set of primes $v$ of $K$ such that at least one of the following conditions is satisfied:
 \begin{itemize}
        \item $v$ is ramified in $L$, 
        \item $E$ has bad reduction at $v$, 
        \item $v$ divides $p$.
    \end{itemize}
\end{definition}  We write $S(L)$ to denote the primes of $L$ that lie above $S$. Let $L_{S}$ be the maximal extension of $L$ that is contained in $\bar{L}$ in which all non-archimedean primes $v\notin S$ are unramified. Then, the Selmer group $\op{Sel}_p(E/L)$ is given by
\begin{equation}\label{Cassels Tate les}\op{Sel}_p(E/L):=\op{ker}\left(H^1(L_S/L, E[p])\xrightarrow{\Phi_{S,L}} \bigoplus_{w\in S(L)} H^1(L_w, E)[p]\right),\end{equation} where $\Phi_{S,L}$ is the direct sum of restriction maps for the primes $w\in S(L)$. For further details, see \cite[section 1, Corollary 6.6]{Milne}. \\

Given a finite abelian group $M$, denote by $M^\vee$ the Pontryagin dual of $M$. Recall the Cassels-Tate long exact sequence
\[\begin{split} & 0\rightarrow \op{Sel}_p(E/L)\rightarrow H^1(L_S/L, E[p])\xrightarrow{\Phi_{S,L}} \bigoplus_{w\in S(L)} H^1(L_w, E)[p] \\ & \rightarrow \op{Sel}_p(E/L)^\vee\rightarrow H^2(L_S/L, E[p])\rightarrow \dots ,\end{split}\]
cf. \cite[p.8]{GCEC}. \\

\noindent Let $G$ be a finite group of $p$-power order and $M$ be a finite $p$-primary $G$-module. Set $M^G:=\{m\in M\mid g m=m \text{ for all }g\in G\}$ and $M_G:=M/\langle (g-1)m \text{ for }g\in G\text{ and }m\in M\rangle$. We record a useful lemma below.  
\begin{lemma}\label{NSW lemma}
If $M^G=0$ or $M_G=0$, then $M=0$.
\end{lemma}
\begin{proof}
The stated result is \cite[Proposition 1.6.12]{NSW}.
\end{proof}

\par Henceforth, we set $L/K$ be a finite Galois extension and set $G:=\op{Gal}(L/K)$. Assume that $G$ is a $p$-group, i.e., $|G|$ is a power of $p$. Let $E$ be an elliptic curve over $K$. Given a prime $v$ of $K$, let $\gamma_v$ be the restriction map 
\begin{equation}\label{def of gamma_v}
   \gamma_v: H^1(K_v, E)[p]\rightarrow \bigoplus_{w|v} H^1(L_w, E)[p]
\end{equation}
where $w$ ranges over all primes of $L$ that lie above $v$.
\begin{proposition}\label{vanishing of selmer propn}
    With respect to notation above, assume that the Selmer group \[\op{Sel}_p(E/K)=0.\] Let $S$ be as in Definition \ref{choice of S}. Then the following assertions hold.
    \begin{enumerate}
        \item[(i)] We have that
\[\dim_{\F_p}\left(\op{Sel}_p(E/L)^G\right)= \sum_{v\in S}\dim_{\F_p}\left(\op{ker}\gamma_v\right).\] 
    \item[(ii)] The $p$-Selmer group $\op{Sel}_p(E/L) = 0$ if and only if $\gamma_v$ is injective for all $v \in S$.
    \end{enumerate}
\end{proposition}
\begin{proof}
Since it is assumed that $\op{Sel}_p(E/K)=0$, the map $\Phi_{S,K}$ in \eqref{Cassels Tate les} is surjective. From the exact sequence \eqref{Selmer TS sequence}, we find that $E(K)\otimes \Z/p\Z=0$. Consequently, $\op{rank} E(K)=0$ and $E(K)[p]=0$. Consider the commutative diagram
    \begin{equation}\label{fdiagram}
\begin{tikzcd}[column sep = small, row sep = large]
& 0=\op{Sel}_p(E/K) \arrow{r}\arrow{d}{\alpha} & H^1(K_S/K, E[p])\arrow{r}{\Phi_{S,K}} \arrow{d}{\beta} & \bigoplus_{v\in S} H^1(K_v, E)[p] \arrow{d}{\gamma} \arrow{r} & 0 \\
0\arrow{r} & \op{Sel}_p(E/L)^G\arrow{r} & H^1(L_S/L, E[p])^G\arrow{r}  &\left(\bigoplus_{w\in S(L)} H^1(L_w, E)[p]\right)^G,
\end{tikzcd}
\end{equation}
where $\gamma:=\bigoplus_{v\in S} \gamma_v$. Since $S$ contains the primes that are ramified in $L$, we can identify $L_S$ with $K_S$. The map $\beta$ is the inflation map from the inflation-restriction sequence. Because the right square in the diagram involving $\beta$ and $\gamma$ commutes, it follows that $\beta$ restricts to a map $\alpha$. From the above diagram, we obtain the long exact sequence
\[0\rightarrow \op{ker}\beta \rightarrow \op{ker}\gamma \rightarrow \op{Sel}_p(E/L)^G\rightarrow \op{cok}\beta\rightarrow \op{cok}\gamma.\]
The inflation restriction sequence now implies that
\[\begin{split}
    \op{ker}\beta\simeq H^1(L/K, E(L)[p]);\\
    \op{cok}\beta\subseteq  H^2(L/K, E(L)[p]).
\end{split}\]
As $L/K$ is a $p$-extension and $E(L)[p]^{G}=E(K)[p]=0$, it follows from Lemma \ref{NSW lemma} that $E(L)[p]=0$. Therefore, $\op{ker}\beta=\op{cok}\beta=0$, and we have an isomorphism 
\[\op{Sel}_p(E/L)^G\simeq \op{ker}\gamma,\] which proves assertion $(i)$. The conclusion $(ii)$ is now evident from Lemma \ref{NSW lemma}.
\end{proof}

We have thus reduced our problem to understanding when the maps $\gamma_v$ are injective for $v \in S$. 
\begin{remark}\label{v-split-injective-remark}
It is clear that if $v$ splits completely in $L$, then $\gamma_v$ is injective. 
\end{remark}
In our applications, we will only consider extensions $L/K$ such that if $v \mid p$ or if $v$ is a prime of bad reduction for $E$, then $v$ splits completely in $L$. Moreover, we will assume henceforth that 
\begin{equation*}
    G \simeq \Z/p^n\Z \qquad \text{ for some } n \in \Z_{\geq 1}.
\end{equation*}
Thus, the remaining section focuses on studying $\op{ker} \gamma_v$ when $v$ is a prime of good reduction, $v \nmid p$ and $v$ is not completely split in $L$. This is established in Proposition \ref{local computations propn}. \\

To this end, we first prove some intermediate results. 
\begin{lemma}\label{vanishing cohomology prop}
    Let $G$ be a finite $p$ group and $M$ be a $G$-module. Assume that $M$ is uniquely divisible by $p$, i.e., the multiplication by $p$ map is an automorphism of $M$. Then $H^i(G, M)=0$ for all $i\geq 1$. 
\end{lemma}
\begin{proof}
    Since $M$ is uniquely divisible by $p$, it follows that multiplication by $p$ on $H^i(G, M)$ is an automorphism. On the other hand, $G$ is a finite $p$-group and hence every element of $H^i(G, M)$ is annihilated by a power of $p$. In greater detail, the restriction co-restriction homomorphism
    \[ H^i(G, M)\xrightarrow{\op{res}} H^i(\{1\}, M)\xrightarrow{\op{cor}} H^i(G, M)\] is multiplication by $|G|$, and hence multiplication by $|G|$ is trivial on $H^i(G, M)$. Therefore, it follows that $H^i(G, M)=0$ for all $i\geq 1$.
\end{proof}

Let $k$ be a finite field and $G$ be a connected smooth algebraic group over $k$. Given a field extension $\mathfrak{l}/k$, let $G(\mathfrak{l})$ denote the $\mathfrak{l}$-rational points of $G$.

\begin{theorem}[Lang]\label{langthm}
   The cohomology group $H^1(\op{Gal}(\bar{k}/k), G(\bar{k}))$ vanishes. 
\end{theorem}
\begin{proof}
    The statement above follows from \cite[Theorem 2]{Lang}. 
\end{proof}

\begin{corollary}\label{langcor}
    Let $k$ be a finite field and $G$ be a connected smooth algebraic group over $k$. Let $\mathfrak{l}/k$ be a finite extension, then, 
    \[H^1(\op{Gal}(\mathfrak{l}/k), G(\mathfrak{l}))=0.\]
\end{corollary}
\begin{proof}
    The inflation map is an injection \[H^1(\op{Gal}(\mathfrak{l}/k), G(\mathfrak{l}))\hookrightarrow H^1\left(\op{Gal}(\bar{k}/k), G(\bar{k})\right)\] and thus $H^1(\op{Gal}(\mathfrak{l}/k), G(\mathfrak{l}))=0$ by Theorem \ref{langthm}.
\end{proof}

\noindent For a prime $w$ of $L$ such that $w \mid v$, let \[W_v:=\op{image}\left\{H^1(L_w/K_v, E(L_w))\xrightarrow{\op{inf}} H^1(K_v, E) \right\}.\] From the inflation-restriction sequence, it is easy to see that 
\begin{equation}\label{ker gamma_v contained in W_v[p]}
\op{ker}\gamma_v\subseteq  W_v[p].
\end{equation} 
Since the inflation map is always injective, we simply identify $W_v$ with $H^1(L_w/K_v, E(L_w))$.\\

\noindent Let $k_w$ (resp. $k_v$) denote the residue field of $L$ at $w$ (resp. $K$ at $v$). Let $\widetilde{E}$ denote the reduction of $E$ at $w$ and $E_0(L_w)\subseteq E(L_w)$ denote the subgroup of points with non-singular reduction. We note that if $w$ is a prime of good reduction, then $E_0=E$. We set $E_1(L_w)$ to denote the kernel of the reduction map
\begin{equation}\label{ses of modules}E_0(L_w)\rightarrow \widetilde{E}(k_w).\end{equation} 
Setting $G_w:=\op{Gal}(L_w/K_v)$, we have an exact sequence of $G_w$-modules
\begin{equation}\label{ses of G_w mods}0\rightarrow E_1(L_w)\rightarrow E_0(L_w)\rightarrow \widetilde{E}(k_w)\rightarrow 0, \end{equation}
where $\widetilde{E}(k_w)$ is unramified.
\begin{proposition}\label{local computations propn}
    Assume that $G\simeq \Z/p^n \Z$ for some $n\in \Z_{\geq 1}$ and that $v\nmid p$ is a prime at which $E$ has good reduction. Let $w$ be any prime of $L$ that lies above $v$. Then the following assertions hold.
    \begin{enumerate}
        \item[(i)]\label{p1 of local comp ramified case} If $v$ is totally ramified in $L$, then
        \[\op{dim}_{\F_p} \left(\op{ker}\gamma_v\right) \leq \dim_{\F_p} \left(\widetilde{E}(k_v)[p]\right).\]
        \item[(ii)]\label{p2 of local comp inert case} If $v$ is unramified and inert in $L$, then $\gamma_v$ is injective. 
        \item[(iii)]\label{p3 local computations propn} If $v$ is ramified in $L$ and $\widetilde{E}(k_v)[p]=0$, then $\gamma_v$ is injective.
    \end{enumerate}
\end{proposition}
\begin{proof}
We first consider the case when $v$ is totally ramified in $L$, and identify $G_w$ with $G$. As $E$ has good reduction at $w$, we find that $E(L_w)=E_0(L_w)$. Since $v\nmid p$, the multiplication by $p$ map acts as an isomorphism on $E_1(L_w)$ and thus, it follows from Lemma \ref{vanishing cohomology prop} that $H^i(G_w, E_1(L_w))=0$ for $i\geq 1$. Thus, from the long exact sequence in cohomology associated to \eqref{ses of modules}, we get the following isomorphism induced by the reduction map
    \begin{equation}\label{redn iso} H^1(G_w, E(L_w))\xrightarrow{\sim} H^1(G_w, \widetilde{E}(k_w)).\end{equation}  
    By equation \eqref{ker gamma_v contained in W_v[p]}, we deduce that $$\dim_{\F_p}(\op{ker}\gamma_v) \leq \dim_{\F_p} W_v[p].$$ Since $L_w/K_v$ is totally ramified, $k_w=k_v$. Therefore, the action of $G_w$ on $\widetilde{E}(k_w)$ is trivial and we find that
    \[\begin{split} \dim_{\F_p} W_v[p]=&\dim_{\F_p} \op{Hom}(G_w, \widetilde{E}(k_w))[p] \\ 
    = &\dim_{\F_p} \op{Hom}(\Z/p^n\Z, \widetilde{E}(k_w)[p]) \\
    =&\dim_{\F_p} \left(\widetilde{E}(k_v)[p]\right).\end{split}\]
    This proves part $(i)$ of the Proposition. \\

    \par Next, we consider the case when $v$ is unramified and inert in $L$. Then we identify $G_w$ with $\op{Gal}(k_w/k_v)$ and from the isomorphism in equation \eqref{redn iso}, we have that 
    \[H^1(G_w, E(L_w))\simeq H^1(k_w/k_v, \widetilde{E}(k_w)).\] The vanishing of $H^1(k_w/k_v, \widetilde{E}(k_w))$ follows from Corollary \ref{langcor}, which is a consequence of Lang's theorem. This proves part $(ii)$ of the Proposition. \\

    In order to prove part (iii), it follows from Lemma \ref{NSW lemma} that $\widetilde{E}(k_{w})[p]=0$. Therefore $\widetilde{E}(k_{w})$ is uniquely $p$-divisible and hence $H^1(G_w, \widetilde{E}(k_{w}))[p]=0$. The isomorphism \eqref{redn iso} does not require $v$ to be totally ramified in $L$ and we find that $H^1(G_w, E(L_w))[p]=0$. Therefore, the kernel of $\gamma_v$ is zero, concluding the proof of part (iii). We thank the referee for suggesting this argument to prove part (iii).
\end{proof}

\begin{proposition}\label{cor 2.7}
    With respect to above notation, assume that the following assertions hold.
    \begin{enumerate}
        \item[(i)] The Selmer group $\op{Sel}_p(E/K)$ is equal to $0$.
        \item[(ii)] All primes of $K$ that lie above $p$ and the primes of bad reduction for $E$ are completely split in $L$.
        \item[(iii)]\label{part 3 E(k_v)[p]} All primes $v$ of $K$ that ramify in $L$ satisfy $\widetilde{E}(k_v)[p]=0$.
    \end{enumerate}
Then, $\op{Sel}_p(E/L)=0$. 
\end{proposition}
\begin{proof}
    Let $S$ be as in Definition \ref{choice of S}. Then it follows from Proposition \ref{vanishing of selmer propn} $(ii)$ that $\op{Sel}_p(E/L)=0$ if and only if $\gamma_v$ is injective for all $v\in S$. For $v\in S$, if $v$ divides $p$ or $v$ is a prime of bad reduction of $E$, then $v$ is assumed to be split completely in $L$, according to hypothesis $(ii)$. By Remark \ref{v-split-injective-remark}, $\gamma_v$ is injective. The only other primes $v\in S$ are those that ramify in $L$. According to our assumption $(iii)$, $\widetilde{E}(k_v)[p] = 0$ for such primes. Then, it follows from Proposition \ref{local computations propn} part $(iii)$ that $\gamma_v$ is injective. This concludes our proof.
    \end{proof}
 \par Recall that $E[p]$ is the $p$-torsion submodule of $E(\overline{\Q})$. Let $\rho_{E, p}:\op{G}_{\Q}\rightarrow \op{GL}_2(\Z/p\Z)$ be the Galois representation on $E[p]$. Let $\Q(E[p])$ be the Galois extension of $\Q$ \emph{cut out} by $\rho_{E,p}$. In other words, it is defined to be the field $\overline{\Q}^{\op{ker}\rho_{E,p}}$. We note that the Galois group $\op{Gal}(\Q(E[p])/\Q)$ can be identified with the image of $\rho_{E,p}$. Given a prime $\ell$ at which $E$ has good reduction, set $\op{G}_\ell:=\op{Gal}(\overline{\Q}_\ell/\Q_\ell)$ and take $\op{I}_\ell\subset \op{G}_\ell$ to be the inertia subgroup, defined to be the kernel of the mod-$\ell$ reduction map 
    \[\op{G}_\ell\rightarrow \op{G}_{\F_\ell}.\] Let $\sigma_\ell\in \op{G}_\ell/\op{I}_\ell$ be the Frobenius automorphism. Denote by $a_\ell(E)$ the trace of the Frobenius element at $\ell$. This is well defined since $\ell$ is a prime of good reduction for $E$. Moreover, the following relations hold (modulo $p$)
    \begin{equation}\label{trace-det-Frob}
        \op{trace} \rho_{E,p}(\sigma_\ell)=a_\ell(E)\text{ and }\op{det} \rho_{E, p}(\sigma_\ell)=\ell.
    \end{equation}
    We set $K(E[p])$ to denote the composite $K\cdot \Q(E[p])$. We introduce a new set of primes which will play a significant role in our results. 
\begin{definition}\label{definition of TE}
    Let $E$ be an elliptic curve over $\Q$ and let $\mathfrak{T}_E$ be the set of prime numbers $\ell$ satisfying the following conditions
    \begin{enumerate}
        \item[(a)] $\ell\neq p$, 
        \item[(b)] $\ell$ is a prime at which $E$ has good reduction, 
        \item[(c)] $\ell$ is completely split in $K(\mu_p)$,
        \item[(d)] $\op{trace}\left(\rho_{E,p}(\sigma_\ell)\right)\neq 2$. 
    \end{enumerate}
\end{definition}
We note that all but finitely many primes $\ell$ satisfy the first two of the above conditions. Given $\ell \in \mathfrak{T}_E$, the condition $(c)$ above implies that $\ell \equiv 1 \bmod p$. It follows from \eqref{trace-det-Frob} that $\rho_{E,p}(\sigma_{\ell}) \in \op{SL}_2(\F_p)$ and from Definition \ref{definition of TE} (d) that $a_{\ell}(E) \not\equiv 2 \bmod p$. Since 
$$a_{\ell}(E) = \ell+1 \,- \, \, \# \widetilde{E}(\F_{\ell}) \qquad \text{and} \qquad \ell \equiv 1 \bmod p,$$ 
we have that $\widetilde{E}(\F_{\ell})[p]=0$ for $\ell \in \mathfrak{T}_E$. The result below gives the relationship between the set of primes $\mathfrak{T}_E$ and the primes $v$ that satisfy the condition $\widetilde{E}(k_v)[p]=0$ from part $(iii)$ of Proposition \ref{cor 2.7}. 

\begin{lemma}\label{T_E lemma}
    Let $\ell$ be a prime in $\mathfrak{T}_E$. Then $\widetilde{E}(k_v)[p]=0$ for any of the primes of $K$ such that $v|\ell$. 
\end{lemma}
\begin{proof}
    Since $\ell$ splits in $K$, we find that $\widetilde{E}(k_v)=\widetilde{E}(\F_\ell)$. That $\widetilde{E}(\F_\ell)[p] = 0$ follows from the discussion above.   
    % Since $\ell$ splits in $\Q(\mu_p)$, it follows that $\ell\equiv 1\mod{p}$. Finally, since $\op{trace}(\rho_{E,p}(\sigma_\ell))\neq 2$, it follows that $a_\ell(E)\not \equiv 2\mod{p}$. It thus follows that 
    % \[\#\widetilde{E}(\F_\ell)=\ell+1-a_\ell(E)\equiv 2-a_\ell(E)\not\equiv 0\mod{p}.\] This proves the result.
\end{proof}

\begin{lemma}\label{density lemma}
    Suppose that $\rho_{E, p}$ is surjective and that \[K(\mu_p)\cap \Q(E[p])=\Q(\mu_p).\] Then the natural density of $\mathfrak{T}_E$ is equal to 
    \[\mathfrak{d}(\mathfrak{T}_E)=\frac{p^2 - p - 1}{[K(\mu_p):\Q(\mu_p)]\,(p^2-1)(p-1)}.\]
\end{lemma}
\begin{proof}
Let $\theta: \op{Gal}(K(E[p])/K) \rightarrow \rho_{E,p}(\op{G}_K)$ be the isomorphism, mapping $\sigma$ to $\rho_{E,p}(\sigma)$. Let $\mathcal{S}$ be the subset of matrices $A$ in $\rho_{E,p}(\op{G}_K)$ with $\op{det} A = 1$ and $\op{trace} A \neq 2$. Let $\ell$ be a prime of good reduction for $E$ with $\ell \neq p$ that is unramified in $K$. Then $\ell$ is unramified in $K(E[p])$. Moreover, $\sigma_{\ell} \in \theta^{-1}( \mathcal{S})$ if and only if $\ell \in \mathfrak{T}_{E}$. Indeed, if $\sigma_{\ell} \in \theta^{-1}(\mathcal{S})$, then $\sigma_{\ell} \in \op{Gal}(K(E[p])/K)$. In fact, $\sigma_{\ell} \in \op{Gal}(K(E[p])/K(\mu_p))$ since $\op{det}(\rho_{E,p}(\sigma_{\ell}))=1$. Therefore $\ell$ splits in $K(\mu_p)$, and $\op{trace}(\rho_{E,p}(\sigma_{\ell})) \neq 2$ by the definition of $\mathcal{S}$, establishing the forward implication. For the converse, if $\ell \in \mathfrak{T}_E$, then it can be argued similarly that $\sigma_{\ell} \in \theta^{-1}(\mathcal{S})$.
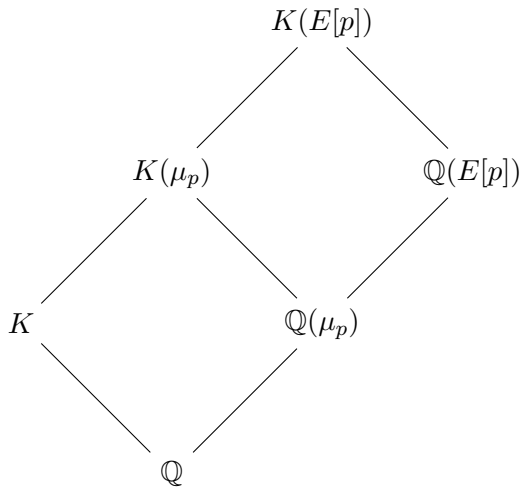
\begin{center}
\begin{tikzpicture}

    \node (Q1) at (0,0) {$\Q$};
    \node (Q2) at (2,2) {$\Q(\mu_p)$};
    \node (Q3) at (4,4) {$\Q(E[p])$};
    \node (Q4) at (-2,2) {$K$};
    \node (Q5) at (0,4) {$K(\mu_p)$};
    \node (Q6) at (2,6) {$K(E[p])$};

    \draw (Q1)--(Q2) node [pos=0.7, below,inner sep=0.25cm] {};
    \draw (Q1)--(Q4) node [pos=0.7, below,inner sep=0.25cm] {};
    \draw (Q2)--(Q3);
    \draw (Q2)--(Q5);
    \draw (Q3)--(Q6);
    \draw (Q4)--(Q5);
    \draw (Q5)--(Q6);
    \end{tikzpicture}
\end{center}
If $\# \mathcal{S}>0$, then by the Chebotarev density theorem, the set of primes that satisfy the above conditions is infinite, with natural density $\left(\frac{\# \mathcal{S}} {[K(E[p]):\Q]}\right)$. We now count the number of matrices $A=\mtx{a}{b}{c}{2-a}\in \op{SL}_2(\F_p)$, that is, the number of triples $(a, b, c)\in \F_p^3$ such that $bc=a(2-a)-1=-a^2+2a-1=-(a-1)^2$. We consider two cases. Suppose $a=1$. Then, either $b$ or $c$ is $0$. Thus, the total number of triples when $a=1$ is $2p-1$. Next assume that $a\neq 1$. Then, $b$ and $c$ are both non-zero, and $c$ is determined by $a$ and $b$. Hence, the total number of triples for which $a\neq 1$ is equal to $(p-1)^2$. Hence, we have  
\begin{align*}
    \# \mathcal{S} & =  \, \# \op{SL}_2(\F_p)-(p-1)^2-(2p-1) = p(p^2-1)-(p-1)^2-(2p-1) \\
    & = p^3 - p^2 - p.
\end{align*}
Hence, we obtain that
\begin{align*}
    \mathfrak{d}(\mathfrak{T}_E) & = \frac{\# \mathcal{S}}{[K(E[p]):K]} = \frac{\# \mathcal{S}}{[K(\mu_p):\Q(\mu_p)][\Q(E[p]):\Q]} = \frac{\#\mathcal{S}}{[K(\mu_p):\Q(\mu_p)] \, \# \op{GL}_2(\F_p)} \\
    & = \frac{p^3 - p^2 - p}{[K(\mu_p):\Q(\mu_p)]\,(p^2-1)(p^2-p)} = \frac{p^2 - p - 1}{[K(\mu_p):\Q(\mu_p)]\,(p^2-1)(p-1)}.
\end{align*}
\end{proof}

\bigskip 

\section{\bf Stability in extensions with Galois group $B\ltimes \Z/p\Z$}\label{s 3}

\par In this section, we fix $p$ to be an odd prime number and $K$ to be a Galois extension of $\Q$, with $B:= \op{Gal}(K/\Q)$. Denote by $\omega$ the mod-$p$ cyclotomic character. Given an $\F_p[\op{G}_{\Q}]$-module $M$ and an integer $k$, we set $M(k)$ to denote the $k$-fold Tate-twist $M\otimes_{\F_p} \F_p(\omega^k)$. Let $E$ be an elliptic curve over $\Q$. Throughout this section, we make the following assumptions
\begin{itemize}
    \item $\op{Sel}_p(E/K)=0$,
    \item $B$ is a finite abelian group, 
    \item $|B|$ is coprime to $p$,
    \item $\chi_0$ is nontrivial mod $p$,
    \item $K$ contains $\Q(\mu_p)$, 
    \item $K\cap \Q(E[p])=\Q(\mu_p)$.
    \item The Galois representation 
    $\rho_{E, p}:\op{G}_{\Q}\rightarrow \op{GL}_2(\F_p)$ is surjective.
\end{itemize}

Let $\G$ be a finite group that is a semi-direct product $\G =B \ltimes T$, where $T \simeq \Z/p\Z$. The more general case when $T \simeq \Z/p^n \Z$ will be the subject of the next section. We consider $n=1$ separately since in this case, the results are stronger and much easier to state. \\

\noindent We call $L/K/\Q$ a $(\G, K)$-extension if 
\begin{itemize}
    \item $L$ is a Galois extension of $\Q$ which contains $K$, 
    \item there is an isomorphism $\phi: \op{Gal}(L/\Q)\xrightarrow{\sim} \G$ which restricts to an isomorphism $\op{Gal}(L/K)\xrightarrow{\sim} T$.
\end{itemize}
We identify the Galois group $\op{Gal}(L/\Q)$ with $\G$ and find that $L^T=K$. We note that the condition $\op{Sel}_p(E/K)=0$ is equivalent to requiring that $\op{rank}E(K)=0$, $E(K)[p]=0$ and $\Sh(E/K)[p]=0$. \\
% The condition that $E(K)[p]=0$ is automatically satisfied since we assume that $\rho_{E, p}$ is surjective and $K\cap \Q(E[p])=\Q(\mu_p)$. In greater detail, the latter conditions imply that $\rho_{E, p}(\op{G}_K)=\op{SL}_2(\F_p)$ and hence in particular, the restriction of $\rho_{E, p}$ to $\op{G}_K$ is irreducible. On the other hand, if $E(K)[p]\neq 0$, then this representation would have a trivial subspace and thus in particular be reducible. Hence, the condition that $E(K)[p]=0$ is a direct consequence of the last two assumptions. 

\par Let $\chi_0: B\rightarrow \op{Aut}(T)\xrightarrow{\sim}\F_p^\times$ be the character defined by 
\[\chi_0(\sigma):=\sigma h \sigma^{-1}.\] The group $\G$ is determined by $\chi_0$ and we write $\G=B \ltimes_{\chi_0} T$ to emphasize this. We take $\chi: \op{G}_{\Q}\rightarrow \F_p^\times$ to be the composite map
\[\op{G}_{\Q}\twoheadrightarrow \op{Gal}(K/\Q)\xrightarrow{\sim} B\xrightarrow{\chi_0} \F_p^\times,\] where the first map above  is the natural quotient map.\\

Let $\F_p(\chi)$ be the $\op{G}_{\Q}$-module on which $\sigma\in \op{G}_{\Q}$, the absolute Galois group, acts as follows: $\sigma\cdot x:=\chi(\sigma) x$, where $\sigma\in \op{G}_{\Q}$ and $x\in \F_p$. Let $Z^1(\op{G}_{\Q}, \F_p(\chi))$ be the group of crossed homomorphisms, $f: \op{G}_{\Q}\rightarrow \F_p(\chi)$. Given a finite set of primes $\Sigma$ containing the primes at which $\chi$ is ramified, let $\Q_\Sigma$ be the maximal algebraic extension of $\Q$ in which the primes $\ell\notin \Sigma$ are unramified. Set $\op{G}_{\Sigma}:=\op{Gal}(\Q_\Sigma/\Q)$ and $Z^1(\Q_\Sigma/\Q, \F_p(\chi)):=Z^1(\op{G}_\Sigma, \F_p(\chi))$. Let $f\in H^1(\Q_\Sigma/\Q, \F_p(\chi))$ and consider the restriction $$g:=f  \big|_{\op{Gal}(\Q_\Sigma/K)}:\op{Gal}(\Q_\Sigma/K)\rightarrow \F_p.$$ 
Since the character $\chi$ is trivial when restricted to $\op{G}_{K}$, the restriction $g$ is a homomorphism.\\
\begin{definition}\label{defn of Kf}Let $K_f$ be the extension of $K$ \emph{cut out} by $f$, i.e., $K_f:=\left(\Q_\Sigma\right)^{\op{ker}g}$.
\end{definition}
The homomorphism $g$ satisfies the relation 
\[g(\sigma x \sigma^{-1})=\chi(\sigma)g(x),\] for $x\in \op{Gal}(\Q_\Sigma/K)$ and $\sigma\in \op{Gal}(\Q_\Sigma/\Q)$. Thus, we find that $\op{ker}(g)$ is a normal subgroup of $\op{Gal}(\Q_\Sigma/\Q)$, and hence, $K_f$ is a Galois extension of $K$. Moreover, the Galois group $\op{Gal}(K_f/\Q)$ is isomorphic to $\G$ and the extension $K_f$ is unramified at all primes $\ell\notin \Sigma$. Since $B$ has order coprime to $p$, $$H^i(B, \F_p(\chi)) = 0 \qquad \text{for} \qquad i > 0. $$ Hence, the restriction map, 
\[\op{res}_K: H^1(\Q_\Sigma/\Q, \F_p(\chi))\rightarrow \op{Hom}\left(\op{Gal}(\Q_\Sigma/K), \F_p(\chi)\right)^B\] is an isomorphism. Therefore, distinct cohomology classes $f, f'$ give rise to distinct homomorphisms $g, g'$. Thus, we find that $K_f=K_{f'}$ if and only if $g=c g'$ for some scalar $c\in \F_p^\times$. This latter condition is equivalent to $f=c f'$.\\

\par Given a finite dimensional $\F_p$-vector space $V$, let $\mathbb{P}(V)$ denote the space of all $\F_p$-lines in $V$. In other words, $\mathbb{P}(V)$ is the set of all equivalence classes $\left(V\backslash \{0\}\right)/\sim$, where $v\sim v'$ if and only if there is a constant $c\in \F_p^\times$ such that $v=cv'$. 
\begin{proposition}\label{bijection 2}
    With respect to the notation above, there is a natural bijection
\[\varphi_\Sigma: \mathbb{P}\left(H^1(\Q_\Sigma/\Q, \F_p(\chi))\right)\xrightarrow{\sim} \{L\mid L\text{ is a }(\G,K)\text{ extension of }\Q\text{, unramified outside }\Sigma\},\] defined by taking $\varphi_\Sigma(f)=K_f$.
\end{proposition}
\begin{proof}
From the inflation restriction sequence, we identify $\op{Hom}\left(\op{Gal}(\Q_\Sigma/K), \F_p(\chi)\right)^B$ with $H^1(\Q_\Sigma/\Q, \F_p(\chi))$. Given a $(\G,K)$-extension $L/K/\Q$, choose an isomorphism $\phi: \op{Gal}(L/\Q)\xrightarrow{\sim} \G$ which restricts to an isomorphism 
\[\phi_{K}:\op{Gal}(L/K)\xrightarrow{\sim} T.\] This in turn gives rise to a non-zero $B$-equivariant isomorphism 
\[\psi:\op{Hom}\left(\op{Gal}(\Q_\Sigma/K), \F_p(\chi)\right)^B,\] which is defined to be the natural composite
\[\op{Gal}(\Q_\Sigma/K)\twoheadrightarrow \op{Gal}(L/K)\xrightarrow{\phi_K} T \xrightarrow{\sim} \F_p(\chi).\] 
Here, the action of $B$ on $\op{Hom}\left(\op{Gal}(\Q_\Sigma/K), \F_p(\chi)\right)$ is as follows. Given $\sigma\in B=\op{Gal}(K/\Q)$ and $g\in \op{Hom}\left(\op{Gal}(\Q_\Sigma/K), \F_p(\chi)\right)$ let $K_g$ be the field extension of $K$ that is cut out by $g$. Pick a lift $\tilde{\sigma}$ of $\sigma$ to $\op{Gal}(K_f/\Q)$. Then, set 
\[(\sigma g)(x):=\chi(\sigma) g(\tilde{\sigma}^{-1} x \tilde{\sigma}).\] This action is well defined, and independent of the choice of lift $\tilde{\sigma}$ since $K_g$ is an abelian extension of $K$. Given two isomorphisms $\phi$ and $\phi'$, the associated maps $\psi$ and $\psi'$ are equivalent. The association $L\mapsto [\psi]$ is the inverse of $\varphi_\Sigma$. This proves the proposition. 
\end{proof}

\par Let $S_{\Q}$ be the set of all rational primes $\ell$ such that at least one of the following conditions hold:
    \begin{itemize}
        \item $\ell$ is ramified in $K$, 
        \item $\ell$ is a prime at which $E$ has bad reduction.
    \end{itemize} 
    We note that $p$ is ramified in $K$, and hence is contained in $S_{\Q}$. 
    %Let $S$ be the set of primes of $K$ that lie above $S_{\Q}$. 
    Let $\Sigma$ be a finite set of rational primes that contains the set $S_{\Q}$. We set $Y:=\Sigma\backslash S_{\Q}$, and assume that $Y\subset \mathfrak{T}_E$, where $\mathfrak{T}_E$ is as in Definition \ref{definition of TE}. We shall introduce splitting conditions for $(\G,K)$-extensions at the primes in $S_{\Q}$. For this purpose, we introduce the relevant Selmer groups. \\
    
\par Denote by $\op{G}_{\Sigma}$ the Galois group $\op{Gal}(\Q_\Sigma/\Q)$. For ease of notation, we set 
\[\begin{split}
& H^i(\Q_\Sigma/\Q, \cdot):=H^i(\op{G}_\Sigma, \cdot),\\
& H^i(\Q_\ell, \cdot):=H^i(\op{Gal}(\overline{\Q}_\ell/\Q_\ell), \cdot).\end{split}\] Let $M$ be a finite $\F_p[\op{G}_\Sigma]$-module. 

\begin{definition}
    A \emph{Selmer structure} supported on $\Sigma$ is a collection of subspaces \[\mathcal{L}_\ell\subseteq H^1(\Q_\ell, M)\] for all $\ell \in \Sigma$. The Selmer group associated to the Selmer structure $\mathcal{L}=\{\mathcal{L}_\ell\}_{\ell\in \Sigma}$ is defined as follows
\[H^1_{\cL}\left(\Q_{\Sigma}/\Q, M\right):=\left\{f\in H^1(\Q_{\Sigma}/\Q, M)\mid \op{res}_\ell(f)\in \cL_\ell\text{ for all }\ell\in \Sigma\right\}.\]
\end{definition}
We note that for an elliptic curve $E$ over $\Q$ and a prime $p$, the $p^n$-Selmer group can be realized in this way for $M:=E[p^n]$. In greater detail, let $\ell$ be a prime number and set $\mathcal{L}_\ell$ to denote the image of the mod-$p$ local Kummer map
\[\kappa_{\Q_\ell,n}: E(\Q_\ell)\otimes \Z/p^n \Z\rightarrow H^1(\Q_\ell, E[p^n]).\] Let $\Sigma$ be the set of primes dividing $Np$ where $N$ is the conductor of $E$. Then we have that 
\[\op{Sel}_{p^n}(E/\Q)=H^1_{\cL}\left(\Q_{\Sigma}/\Q, E[p^n]\right).\]
Given $f\in H^1_{\cL}(\Q_\Sigma/\Q, \F_p(\chi))$ recall from Definition \ref{defn of Kf} that $K_f$ is the extension of $K$ cut out by $f$. A $(\G, K)$-extension $L/K/\Q$ is said to satisfy the Selmer condition $\cL$ if $L=K_f$ for some $f\in H^1_{\cL}(\Q_\Sigma/\Q, \F_p(\chi))$. When $\cL_\ell=0$ (resp. $\cL_\ell=H^1(\Q_\ell, \F_p(\chi))$) we say that the condition at $\ell$ is \emph{strict} (resp. \emph{relaxed}). Throughout, we shall impose the strict condition above the primes $\ell\in S_{\Q}$ and the relaxed condition above the primes $\ell\in Y$. More precisely, we specify 
\begin{equation}\label{selmer-structure}
    \cL_{\ell}:=\begin{cases}
    0 & \text{ if }\ell \in S_{\Q};\\
    H^1(\Q_\ell, \F_p(\chi)) & \text{ if }\ell\in Y.
\end{cases}
\end{equation}

\noindent We denote the associated Selmer group by 
\[H^1_{Y}(\Q_\Sigma/\Q, \F_p(\chi)):=H^1_{\cL}(\Q_\Sigma/\Q, \F_p(\chi)),\] 
where $\mathcal{L}$ is as defined in \eqref{selmer-structure}. We fix this choice of $\mathcal{L}$ for the rest of this section.
\begin{definition}\label{W-Y}
    Let $\mathcal{W}_Y=\mathcal{W}_Y(E,K)$ be the set of  $(\G,K)$-extensions which are unramified outside $\Sigma$ and completely split at primes $v$ of $K$ above $S_{\Q}$.
\end{definition}
If $Y'\subseteq Y$, then, $\mathcal{W}_{Y'}\subseteq \mathcal{W}_Y$. The classes in $H^1_{Y}(\Q_\Sigma/\Q, \F_p(\chi))$ give rise to extensions in $\mathcal{W}_Y$.\\

\begin{proposition}\label{prop sel(e/l)=0}
    With respect to notation above, suppose that $L\in \mathcal{W}_Y$. Assume that $\op{Sel}_p(E/K)=0$. Then $\op{Sel}_p(E/L)=0$. 
\end{proposition}
\begin{proof}
It suffices to show that the conditions of Proposition \ref{cor 2.7} are satisfied. \begin{enumerate}
    \item[(i)] We have assumed that $\op{Sel}_p(E/K)=0$. 
    \item[(ii)] Since $L\in \mathcal{W}_Y$, we find that all primes of $K$ above $S_{\Q}$ are completely split in $L$. In particular, all primes of $K$ that lie above $p$ and the primes of bad reduction for $E_{/K}$ are completely split in $L$. 
    \item[(iii)] The set of primes of $K$ that ramify in $L$ is a subset of the primes that lie above $Y$. Since $Y$ is assumed to be a subset of $\mathfrak{T}_E$, it follows from Lemma \ref{T_E lemma} that $\widetilde{E}(k_v)[p]=0$ for all primes of $K$ that lie above $Y$.
\end{enumerate}
It thus follows from Proposition \ref{cor 2.7} that $\op{Sel}_p(E/L)=0$. 
\end{proof}

We now recall the formula due to Wiles that allows one to compute the size of a Selmer group by relating it to the dual Selmer group. Let $M$ be a finite $\F_p[\op{G}_\Sigma]$-module and $\cL$ be a Selmer structure on $\Sigma$. Let $M^\vee:=\op{Hom}_{\F_p}\left(M, \F_p(1)\right)$ be the Tate-dual of $M$. At each prime $\ell\in \Sigma$, let $\cL_\ell^\vee$ be the orthogonal complement of $\cL_\ell$ with respect to the nondegenerate local Tate pairing
\[H^1(\Q_\ell, M)\times H^1(\Q_\ell, M^\vee)\rightarrow H^2(\Q_\ell, \F_p(1))\xrightarrow{\sim} \F_p.\]

\begin{definition}
    The \emph{dual Selmer group} is defined as follows
\[H^1_{\cL^\vee} (\Q_\Sigma/\Q, M^\vee):=\left\{g\in H^1(\Q_{\Sigma}/\Q, M^\vee)\mid \op{res}_\ell(g)\in \cL_\ell^\vee\text{ for all }\ell\in \Sigma\right\} .\]
\end{definition}
For $M:=\F_p(\chi)$, the Tate-dual is $M^\vee=\F_p(\omega \chi^{-1})$. For a choice of $Y$ and the $S_{\Q}$-strict Selmer structure, the dual Selmer group in our context is given by
\[H^1_{Y} (\Q_\Sigma/\Q, \F_p(\omega\chi^{-1})):=\left\{g\in H^1(\Q_{\Sigma}/\Q, \F_p(\omega\chi^{-1}))\mid \op{res}_\ell(g)=0\text{ for all }\ell\in Y\right\} .\]

\begin{theorem}[Wiles]\label{wiles thm}
    Let $M$ be a finite $\F_p[\op{G}_\Sigma]$-module. The formula of Wiles asserts that 
    \[\frac{\# H^1_{\cL} (\Q_\Sigma/\Q, M)}{\# H^1_{\cL^\vee} (\Q_\Sigma/\Q, M^\vee)}=\left(\frac{\#H^0(\Q, M)}{\# H^0(\Q, M^\vee)}\right)\times \prod_{\ell\in \Sigma} \left(\frac{\# \cL_\ell}{\# H^0(\Q_\ell, M)}\right).\]
\end{theorem}
\begin{proof}
    For a proof of the result, cf. \cite[Ch. 8]{NSW}.
\end{proof}

\begin{proposition}\label{asymptotic of WY}
    With respect to above notation, there is an explicit constant $C>0$, that depends only on $S_{\Q}$ and $K$, such that 
    \[\# \mathcal{W}_Y\geq \, \frac{p^{|Y|}}{C}-\frac{1}{(p-1)}.\]
\end{proposition}
\begin{proof}
It follows from the proof of Proposition \ref{bijection 2} that the set $\mathcal{W}_Y$ is in bijection with $\mathbb{P}\left(H^1_Y(\Q_\Sigma/\Q, \F_p(\chi))\right)$ and hence, 
\begin{equation}\label{H-1-card-eqn}
    \# \mathcal{W}_Y= \, \frac{\# H^1_Y(\Q_\Sigma/\Q, \F_p(\chi))-1}{p-1}.
\end{equation}
The formula of Wiles, as in Theorem \ref{wiles thm}, implies that 
\[\# H^1_Y(\Q_\Sigma/\Q, \F_p(\chi))=p\times \# H^1_Y(\Q_\Sigma/\Q, \F_p(\omega\chi^{-1}))\times C^{-1}\times \prod_{\ell\in Y} \left(\frac{\# H^1(\Q_\ell, \F_p(\chi))}{\# H^0(\Q_\ell, \F_p(\chi))}\right),\] where
\[C:=p\times\#H^0(\Q, \F_p(\omega\chi^{-1}))\times \prod_{\ell\in S_{\Q}} \# H^0(\Q_\ell, \F_p(\chi));\] a quantity that depends only on $S_{\Q}$ and $\chi$, and not on $Y$. \\

\par Note that each of the primes $\ell\in Y$ is in $\mathfrak{T}_E$ and hence splits in $K$. Therefore the character $\chi$ is trivial when restricted to $\op{G}_{\Q_\ell}$. Also, we note that since $\Q(\mu_p)$ is contained in $K$, the prime $\ell$ is $1\bmod p$. It follows from class field theory that $\# H^1(\Q_\ell, \F_p)=p^2$ and therefore, 
\[\frac{\# H^1(\Q_\ell, \F_p(\chi))}{\# H^0(\Q_\ell, \F_p(\chi))}=\frac{p^2}{p}=p\] for $\ell\in Y$. Therefore, substituting the above in \eqref{H-1-card-eqn}, we have that
\[\# H^1_Y(\Q_\Sigma/\Q, \F_p(\chi))\geq \frac{p^{|Y|+1}}{C}\]
and
 \[\# \mathcal{W}_Y\geq \left(\frac{\frac{p^{|Y|+1}}{C}-1}{p-1}\right)\geq \frac{p^{|Y|}}{C}-\frac{1}{(p-1)}.\]
This completes the proof. 
\end{proof}

\noindent Thus, we now have our main result of this section.
\begin{theorem}\label{n=1 thm}
Let $K$ be a Galois extension of $\Q$ and $p$ be an odd prime number. Let $E$ be an elliptic curve over $\Q$ and assume that the following conditions are satisfied
\begin{enumerate}
    \item[(i)] $\op{Sel}_p(E/K)=0$,
    \item[(ii)] $B:=\op{Gal}(K/\Q)$ is a finite abelian group, 
    \item[(iii)] $|B|$ is prime to $p$,
    \item[(iv)] $K$ contains $\Q(\mu_p)$, 
    \item[(v)] $K\cap \Q(E[p])=\Q(\mu_p)$.
    \item[(vi)] The Galois representation 
    $\rho_{E, p}:\op{G}_{\Q}\rightarrow \op{GL}_2(\F_p)$ is surjective.
\end{enumerate}
Let $\G$ be a finite group that is a semi-direct product $\G=B\ltimes T$, where $T\simeq \Z/p\Z$. Let $\mathfrak{T}_E$ be as in Definition \ref{definition of TE}, and for any finite set $Y \subset \mathfrak{T}_E$, let $\mathcal{W}_Y$ be as in Definition \ref{W-Y}. Then 
    \[\# \mathcal{W}_Y\gg p^{|Y|},\]
    for any finite set $Y\subset \mathfrak{T}_E$. 
    Moreover, for any $L\in \mathcal{W}_Y$, the Selmer group $\op{Sel}_p(E/L)=0$. In particular, the rank of $E(L)$ is $0$.
\end{theorem}
Note that $\mathfrak{T}_E$ is an infinite set of primes with density
$$ \mathfrak{d}(\mathfrak{T}_E)=\frac{p^2-p-1}{[K:\Q(\mu_p)](p-1)(p^2-1)}, $$
as proved in Lemma \ref{density lemma}, with $ \Q(\mu_p) \subseteq K $.
\begin{proof}[Proof of Theorem \ref{n=1 thm}]
Recall from Definition \ref{definition of TE} that $\mathfrak{T}_E$ is set of prime numbers $\ell$ satisfying the following conditions
    \begin{itemize}
        \item $\ell\neq p$, 
        \item $\ell$ is a prime at which $E$ has good reduction, 
        \item $\ell$ is completely split in $K(\mu_p)$,
        \item $\op{trace}\left(\rho_{E,p}(\sigma_\ell)\right)\neq 2$. 
    \end{itemize}
It follows from Proposition \ref{asymptotic of WY} that $\#\mathcal{W}_Y\gg p^{|Y|}$, and from Proposition \ref{prop sel(e/l)=0} that $\op{Sel}_p(E/L)=0$ for $L \in \mathcal{W}_Y$.
\end{proof}

\subsection*{An Example} We end this section with an example where the conditions of Theorem \ref{n=1 thm} are satisfied. We take $p=3$ and since $K=\Q(\mu_3)$. The cyclotomic character $\omega:\op{Gal}(K/\Q)\rightarrow (\Z/3\Z)^\times$ is an isomorphism, and $\chi_0: (\Z/3\Z)^\times\rightarrow (\Z/3\Z)^\times$ is the identity map. Then, $\chi$ is equal to $\omega$. Consider the elliptic curve 
\[E: y^2+y=x^3-x^2-7820x-263580;\] consider the following code on Sage
\begin{verbatim}
E=EllipticCurve("11a2")
K = QuadraticField(-3, 'z')
EK = E.base_extend(K)
EK.rank()
rho = E.galois_representation()
rho.image_type(3)
EK.torsion_order()
\end{verbatim}
which shows that $E(K)[3]=0$, $\op{rank}E(K)=0$ and $\rho_{E,3}$ is surjective. It remains to show that $\Sh(E/K)[3]=0$ and then we shall deduce that $\op{Sel}_3(E/K)=0$. 

\noindent Note that 
\[\op{Sel}_3(E/K)=\op{Sel}_3(E/\Q)\oplus \op{Sel}_3(E^{(-3)}/\Q)\]
where $E^{(-3)}$ is the quadratic twist of $E$ by $-3$. Thus, it suffices to show (via $3$-descent) that $\op{Sel}_3(E/\Q)=0$ and $\op{Sel}_3(E^{(-3)}/\Q)=0$. For this, we use the following code on Magma
\begin{verbatim}
E := EllipticCurve("11a2");
ThreeSelmerGroup(E);
Et:=QuadraticTwist(E, -3);
ThreeSelmerGroup(Et);
\end{verbatim}
which shows that $\op{Sel}_3(E/\Q)$ and $\op{Sel}_3(E^{(-3)}/\Q)$ are trivial.

\bigskip 

\section{\bf Stability in extensions with Galois group $B\ltimes \Z/p^n\Z$}\label{s 4}
\par In this section, $p\geq 5$ is a prime number. We consider $K/\Q$ and $E_{/\Q}$ with $B:=\op{Gal}(K/\Q)$. We let $\G =B \ltimes T$ where $T \simeq \Z/p^n \Z$ and $\chi_0:B\rightarrow (\Z/p^n \Z)^\times$ be the associated character. We take $\chi: \op{G}_\Q \rightarrow (\Z/p^n \Z)^\times$ the associated Galois character which factors as 
\[\op{G}_\Q\rightarrow \op{Gal}(K/\Q)\xrightarrow{\sim} B\xrightarrow{\chi_0} (\Z/p^n \Z)^\times.\]
Throughout, $M$ will denote the Galois module $(\Z/p^n \Z)(\chi)$. We set $\epsilon_n$ to denote the mod-$p^n$ cyclotomic character and identify the Tate-dual of $M$ with $(\Z/p^n \Z) (\epsilon_n \chi^{-1})$. \\

Assume that the following conditions are satisfied
\begin{itemize}
    \item $\op{Sel}_p(E/K)=0$,
    \item $B$ is a finite abelian group,
    \item $K$ contains $\Q(\mu_{p^n})$. In particular, the group $B$ has order divisible by $p^{n-1}(p-1)$.
    \item The character $\chi_0$ is nontrivial modulo $p$.
    \item The intersection $K\cap \Q(E[p])$ is equal to $\Q(\mu_p)$.
    \item The Galois representation 
    $\rho_{E, p}:\op{G}_{\Q}\rightarrow \op{GL}_2(\F_p)$ is surjective.
\end{itemize}
Once more, let $S_{\Q}$ be the set of all rational primes $\ell$ such that at least one of the following conditions hold:
    \begin{itemize}
        \item $\ell$ is ramified in $K$,
        \item $\ell$ is a prime at which $E$ has bad reduction.
    \end{itemize} 
For a finite set of primes $\Sigma$ that contains $S_{\Q}$, set $Y:=\Sigma\backslash S_{\Q}$. Let $\mathfrak{T}_E$ be as in Definition \ref{definition of TE}. Assume that $Y \subset \mathfrak{T}_E$. \\

Consider the inflation restriction sequence
\begin{equation}\label{some inflation-res seq}
    0\rightarrow H^1(B, M)\xrightarrow{\op{inf}_K} H^1(\Q_\Sigma/\Q, M)\xrightarrow{\op{res}_K} \op{Hom}\left(\op{Gal}(\Q_\Sigma/K), M\right)^B.
\end{equation}
We take $H^1_{\op{surj}}(\Q_\Sigma/\Q, M)$ to be the subset of $H^1(\Q_\Sigma/\Q, M)$ consisting of classes for which $\op{res}_K(f)$ surjects onto $M$. We note that when $n=1$, the set $H^1_{\op{surj}}(\Q_\Sigma/\Q, M)$ consists of non-zero classes. Given a cohomology class $f\in H^1_{\op{surj}}(\Q_\Sigma/\Q, M)$, set $g:=\op{res}_K(f)$ and let $K_f:=\overline{\Q}^{\op{ker} g}$. It is easy to see that $K_f$ is a $(\G,K)$-extension of $\Q$. \\

Define $f\sim f'$ if there exists $c\in (\Z/p^n\Z)^\times$ such that $f-c f'\in \op{inf}_K\left(H^1(B, M)\right)$, and let $\overline{H^1}_{\op{surj}}(\Q_\Sigma/\Q, M)=H^1_{\op{surj}}(\Q_\Sigma/\Q, M)/\sim$ denote the set of equivalence classes. If $f \sim f'$, then $K_f = K_{f'}$. It is also easy to see that 
\[\# \overline{H^1}_{\op{surj}}(\Q_\Sigma/\Q, M)=\frac{\# H^1_{\op{surj}}(\Q_\Sigma/\Q, M)}{p^{n-1}(p-1)\, \, \# H^1(B, M)}.\]

\begin{proposition}\label{bijection 1}
    With respect to notation above, there is a natural injection
    \[\overline{H^1}_{\op{surj}}(\Q_\Sigma/\Q, M) \hookrightarrow \{L\mid L\text{ is a }(\G,K)\text{ extension of }\Q\text{, unramified outside }\Sigma\},\] defined by taking $\varphi_\Sigma(f):=K_f$.
\end{proposition}
\begin{proof}
    From \eqref{some inflation-res seq}, identify $\overline{H^1}_{\op{surj}}(\Q_\Sigma/\Q, M)$ with the equivalence classes of $B$-equivariant surjective homorphisms $g:\op{Gal}(\Q_\Sigma/K) \rightarrow M$. Here, $g\sim g'$ if there exists $c\in (\Z/p^n \Z)^\times$ such that $g=c g'$. These equivalence classes give rise to the set of $(\G,K)$-extensions that are unramified outside $\Sigma$, where the association in question is $\varphi_\Sigma$.
\end{proof}

\par In this section, we write $Y= V \cup Z$, where $V$ and $Z$ are disjoint sets. The set $Z$ is be fixed throughout and will be chosen later. The set $V$ will vary over non-empty finite subsets of $\mathfrak{T}_E$ that are disjoint from $Z$. 
\begin{definition}\label{B-V definition}
    Let $\mathfrak{B}_V =\mathfrak{B}_V(E,K)$ denote the set of $(\G,K)$-extensions $L/K/\Q$ such that
\begin{itemize}
    \item primes outside $\Sigma=S_{\Q}\cup V\cup Z$ are unramified in $L$ 
    \item primes of $K$ above $V$ are totally ramified in $L$
    \item primes of $K$ above $S_{\Q}$ are completely split in $L$.
\end{itemize}
\end{definition}

As in the previous section, we wish to parameterize such extensions by Selmer classes. Given a prime $\ell\in Y$, since $Y$ is a subset of $\mathfrak{T}_E$, $\ell$ splits in $K$. Since $K$ contains $\Q(\mu_{p^n})$ by assumption, it follows that $\ell\equiv 1\mod{p^n}$. As $\chi$ is trivial when restricted to $\op{G}_K$, the character $\chi$ is also trivial when restricted to $\op{G}_{\Q_\ell}$. By local class field theory, $H^1(\Q_\ell, M)= \op{Hom}\left(\op{G}_{\Q_\ell}, M\right)$, and thus, $\# H^1(\Q_\ell, M)=p^{2n}$. For a prime $\ell\in V$, and a class $f\in H^1_Y(\Q_\Sigma/\Q, M)$, we say that $f$ is totally ramified at $\ell$ if the restricted homomorphism $\op{res}_\ell(f):\op{I}_\ell\rightarrow M$ is surjective. \\

\par Let $\cL$ be the Selmer structure supported on $\Sigma$ that is strict at all primes of $S_{\Q}$ and relaxed at all primes of $Y$. More specifically
\[\cL_{\ell}:=\begin{cases}
    0 & \text{ if }\ell \in S_{\Q};\\
    H^1(\Q_\ell, M) & \text{ if }\ell\in Y.
\end{cases}\]
As in the previous section, we let $H^1_Y(\Q_\Sigma/\Q, M)$ and $H^1_{Y}(\Q_\Sigma/\Q, M^\vee)$ be the associated Selmer and dual Selmer groups respectively. We have 
\begin{align*}
    H^1_Y(\Q_\Sigma/\Q, M) & =\{f\in H^1(\Q_\Sigma/\Q, M)\mid f\text{ is trivial at all primes }\ell\in S_{\Q}\};\\
    H^1_Y(\Q_\Sigma/\Q, M^\vee) & =\{\psi\in H^1(\Q_\Sigma/\Q, M^\vee)\mid \psi\text{ is trivial at all primes }\ell\in Y\}.
\end{align*}
As a consequence of Wiles' formula (Theorem \ref{wiles thm}), we get
\[\begin{split}\frac{\# H^1_Y (\Q_\Sigma/\Q, M)}{ \# H^1_Y (\Q_\Sigma/\Q, M^\vee)}=& \left(\frac{\# H^0(\Q, M)}{\# H^0(\Q, M^\vee)}\right)\times \left(\prod_{\ell\in S_{\Q}} \frac{1}{\# H^0(\Q_\ell, M)}\right)\times \left(\prod_{\ell\in Y} \frac{\# H^1(\Q_\ell, M)}{\# H^0(\Q_\ell, M)}\right) \\
=& \frac{p^{n|Y|}}{C},\end{split}\] where $C$ is a constant that does not depend on $\Sigma$. We want to count the number of classes in $H^1_Y (\Q_\Sigma/\Q, M)$ that are totally ramified at the primes $\ell\in V$. In order to make this possible, we choose a set of primes $Z\subset \mathfrak{T}_E$, so that the dual Selmer group $H^1_Y(\Q_\Sigma/\Q, M^\vee)=0$ for any choice of $V \subset \mathfrak{T}_E$. \\

Since $\Q(\mu_{p^n}) \subseteq K$, $\epsilon_n$ is trivial when restricted to $G_K$. The same is also true of $\chi$, as is evident from its definition. Therefore, the action of $G_K$ on $M^\vee$ is trivial. Consider the following inflation-restriction sequence.
\begin{equation}\label{another inflation-rest sequence}
    0 \rightarrow H^1(B,M^\vee) \xrightarrow{\op{inf}_K} H^{1}(\Q_{\Sigma}/\Q, M^{\vee}) \xrightarrow{\op{res}_K} \op{Hom} {\left( \op{Gal}(\Q_{\Sigma}/K), M^\vee \right)}^B.
\end{equation}
Let $0 \neq \psi \in H^1( \Q_{\Sigma}/\Q, M^\vee)$, and $\psi_K := \op{res}_K(\psi)$ be the restriction of $\psi$ to $G_K$. Define $K_{\psi}:= \overline{K}^{\op{ker}(\psi_K)}$ to be the extension of $K$ that is cut out by $\psi$. Note that $K_{\psi}/K$is a Galois extension, with $\op{Gal}(K_{\psi}/K) \simeq \op{image}(\psi_K)$. Hence, $\op{Gal}(K_{\psi}/K) \simeq \Z/p^m\Z$ for some $m \leq n$. \\

The proof of the following Lemma requires the assumption that $p \geq 5$. We use the fact that the group $\op{PSL}_2(\F_p)$ is known to be simple when $p \geq 5$. Given a finite Galois extension $F/\Q$, by a Chebotarev class in $\op{Gal}(F/\Q)$ we mean a conjugacy class in $\op{Gal}(F/\Q)$.
\begin{lemma}\label{technical lemma 1}
Let $\psi\in H^1(\Q_\Sigma/\Q, M^\vee)$ be a class for which $\psi_K$ is non-trivial. Let $\mathcal{C}_{\psi}$ be the set of all rational primes $\ell$ such that 
   \begin{enumerate}
       \item[(a)] $\ell$ splits in $K$, 
       \item[(b)] the trace of $\rho_{E,p}(\sigma_\ell)$ is not $2 \bmod p$, 
       \item[(c)] any prime $v$ of $K$ that lies above $\ell$ does not split completely in $K_{\psi}$.
   \end{enumerate}
Then the collection of primes $\ell\in \mathcal{C}_{\psi}$ is infinite and defined by Chebotarev classes in the Galois group $\op{Gal}(K_{\psi}(E[p])/\Q)$.
\end{lemma}
\begin{proof}
Let $\ell$ be a rational prime, unramified in $K$. Denote by $\mathcal{C}_1 \subseteq \op{Gal}(K/\Q)$ the conjugacy class consisting of the identity element. The condition (a) is satisfied if and only if  $\sigma_{\ell} \in \mathcal{C}_1$. Now, view $\mathcal{C}_1 \subseteq \op{Gal}(K/\Q(\mu_p))$. Since $\rho_{E,p}$ is surjective, we identify $\op{Gal}(\Q(E[p])/\Q(\mu_p))$ with $\op{SL}_2(\F_p)$. Denote by $\mathcal{C}_2 \subseteq \op{SL}_2(\F_p)$ the matrices with trace not equal to $2$, that is condition (b). Recall that $K\cap \Q(E[p])=\Q(\mu_p)$, and hence, 
\begin{equation*}
    \op{Gal}(K(E[p]) / \Q(\mu_p)) \simeq \op{Gal}(K/\Q(\mu_p)) \times \op{Gal}(\Q(E[p])/\Q(\mu_p)).
\end{equation*}
Define $\mathcal{C}_3 := \mathcal{C}_1 \times \mathcal{C}_2$ and view it as a subset of $ \op{Gal}(K(E[p]) / \Q)$. Thus, both conditions (a) and (b) are satisfied if and only if $\ell$ is unramified in $K(E[p])$ and $\sigma_{\ell} \in \mathcal{C}_3$. \\

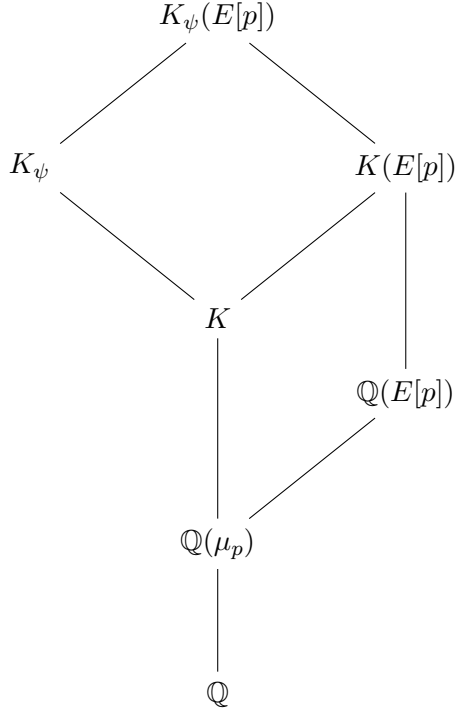
\begin{center}
\begin{tikzpicture}

    \node (Q1) at (0,0) {$\Q$};
    \node (Q2) at (0,2) {$\Q(\mu_p)$};
    \node (Q3) at (0,5) {$K$};
    \node (Q4) at (2.5,4) {$\Q(E[p])$};
    \node (Q5) at (2.5,7) {$K(E[p])$};
    \node (Q6) at (-2.5,7) {$K_\psi$};
    \node (Q7) at (0,9) {$K_\psi(E[p])$};

    \draw (Q1)--(Q2);
    \draw (Q2)--(Q3);
    \draw (Q3)--(Q5);
    \draw (Q2)--(Q4);
    \draw (Q6)--(Q7);
    \draw (Q5)--(Q7);
    \draw (Q4)--(Q5);
    \draw (Q3)--(Q6);
    \end{tikzpicture}
\end{center}

\noindent Note that $K_{\psi} \cap K(E[p]) = K$ as $p \geq 5$. In greater detail, $K_{\psi}/K$ is a cyclic $p$-extension, whereas $\op{Gal}(K(E[p])/K)\simeq \op{Gal}(\Q(E[p])/\Q(\mu_p)) \simeq \op{SL}_2(\F_p)$. As $\op{PSL}_2(\F_p)$ is simple when $p \geq 5$, there are no sub-extensions $K \subsetneq K' \subsetneq  K(E[p])$ with $K'/K$ being cyclic, of order dividing $p$. Consequently, $K_{\psi} \, \cap \, K(E[p]) = K$, and
\begin{equation*}
    \op{Gal}(K_{\psi}(E[p])/K) \simeq \op{Gal}(K_{\psi}/K) \times \op{Gal}(K(E[p]/K).
\end{equation*}
Since $\mathcal{C}_1$ is identity, $\mathcal{C}_3 \subseteq \op{Gal}(K(E[p])/K)$. Let $\mathcal{C}_4 \subseteq \op{Gal}(K_{\psi}/K)$ consist of the non-trivial elements, and define $\mathcal{C}_{\psi} := \mathcal{C}_3 \times \mathcal{C}_4 $, viewed as a subset of $\op{Gal}(K_{\psi}(E[p])/K) \subseteq \op{Gal}(K_{\psi}(E[p])/\Q)$. Thus, for a rational prime $\ell$, unramified in $K$, the conditions (a), (b) and (c) are satisfied if and only if $\sigma_{\ell} \in \mathcal{C}_{\psi}$. 
\end{proof}

We now apply the above lemma to show the following.
\begin{proposition}\label{tech prop 1}
There exists a finite set of primes $Z \subset \mathfrak{T}_E$ such that $$H^{1}_Y( \Q_{\Sigma}/\Q, M^\vee) = \op{inf}_K \left( H^1(B, M^\vee) \right)$$
whenever $Z \subseteq Y$.
\end{proposition}
\begin{proof}
    \par We recall that \[H^1_Y(\Q_\Sigma/\Q, M^\vee)=\{f\in H^1(\Q_\Sigma/\Q, M^\vee)\mid \op{res}_\ell(f)=0\text{ for all }\ell \in Y\}.\] Since all primes of $Y$ split in $K$, it follows that $\op{inf}_K\left(H^1(B, M^\vee)\right)$ is contained in the above dual Selmer group.
    Let $\psi$ be a cohomology class in $H^1(\Q_\Sigma/\Q, M^\vee)$ that is not contained in the image of the inflation map. Then, $\psi_K \neq 0$ and and it follows from Lemma \ref{technical lemma 1} that there is an infinite collection of primes $\ell$ in $\mathfrak{T}_E$ such that $\op{res}_\ell(f)\neq 0$. We enumerate the classes $\psi \in H^1(\Q_{\Sigma}/\Q, M^\vee)$ such that $\psi_{K}$ is non-trivial as $\psi_1$, $\psi_2$, $\cdots$, $\psi_m$. For each $\psi_j$, choose a prime $\ell_j \in \mathcal{C}_{\psi_j}$ and $Z := \{\ell_1, \ell_2, \cdots, \ell_m \}$. For this choice of $Z$, we have that $H^1_Y(\Q_\Sigma/\Q, M^\vee)$ is equal to $\op{inf}_K\left(H^1(B, M^\vee)\right)$.
\end{proof}

\begin{lemma}\label{prop 4.4}
    Let $Y$ and $Z$ be as in Proposition \ref{tech prop 1} and $\Sigma = V \cup Z\cup S_{\Q}$. Then, we have
    \begin{equation}\label{first eqn}\# H^1_Y(\Q_\Sigma/\Q, M^\vee)=\#H^1_{Z}(\Q_{Z\cup S_{\Q}}/\Q, M^\vee). \end{equation}Moreover, 
    \[\# H^1_{Y}(\Q_{\Sigma}/\Q, M)=\# H^1_{Z}(\Q_{Z\cup S_{\Q}}/\Q,M^\vee)\times p^{n|V|}.\]
\end{lemma}
\begin{proof}
   Proposition \ref{tech prop 1} implies that 
    \[H^1_Y(\Q_\Sigma/\Q, M^\vee)=H^1_{Z}(\Q_{Z\cup S_{\Q}}/\Q, M^\vee)=\op{inf}_K\left(H^1(B, M^\vee)\right), \] and thus equation \eqref{first eqn} is proved. It follows from Wiles' formula (Theorem \ref{wiles thm}) that
    \[\begin{split}\frac{\# H^1_Y (\Q_\Sigma/\Q, M)}{ \# H^1_Y (\Q_\Sigma/\Q, M^\vee)}=& \left(\frac{\# H^0(\Q, M)}{\# H^0(\Q, M^\vee)}\right)\times \left(\prod_{\ell\in S_{\Q}} \frac{1}{\# H^0(\Q_\ell, M)}\right)\times \left(\prod_{\ell\in Y} \frac{\# H^1(\Q_\ell, M)}{\# H^0(\Q_\ell, M)}\right) ;\\
\frac{\# H^1_{Z} (\Q_{S_{\Q}\cup Z}/\Q, M)}{ \# H^1_Z (\Q_{S_{\Q}\cup Z}/\Q, M^\vee)}=& \left(\frac{\# H^0(\Q, M)}{\# H^0(\Q, M^\vee)}\right)\times \left(\prod_{\ell\in S_{\Q}} \frac{1}{\# H^0(\Q_\ell, M)}\right)\times \left(\prod_{\ell\in Z} \frac{\# H^1(\Q_\ell, M)}{\# H^0(\Q_\ell, M)}\right).\\
\end{split}\]
From equation \eqref{first eqn}, we find that 
\[\begin{split}\# H^1_Y (\Q_\Sigma/\Q, M)=& \# H^1_{Z} (\Q_{S_{\Q}\cup Z}/\Q, M)\times \left(\prod_{\ell\in V} \frac{\# H^1(\Q_\ell, M)}{\# H^0(\Q_\ell, M)}\right); \\ 
= &\# H^1_{Z} (\Q_{S_{\Q}\cup Z}/\Q, M)\times p^{n|V|}.\end{split}\]
\end{proof}

Let $\ell\in \mathfrak{T}_E$. Since $\ell$ splits completely in $K$, which contains $\Q(\mu_{p^n})$, the Galois action of $\op{G}_{\Q_\ell}$ on $M$ is trivial. Let $H^1_{\op{nr}}(\Q_\ell, M)$ be the image of the inflation map 
\[\op{inf}_\ell: H^1(\op{G}_{\Q_\ell}/\op{I}_\ell, M)\rightarrow H^1(\Q_\ell, M).\]
Since the Galois action on $M$ is trivial, $H^1_{\op{nr}}(\Q_\ell, M)$ is identified with the subgroup of unramified homomorphisms in $\op{Hom}(\op{G}_{\Q_\ell}, \Z/p^n\Z)$, and thus, 
\begin{equation}\label{size of unram classes}
    \# H^1_{\op{nr}}(\Q_\ell, M)=p^n.
\end{equation}
\begin{proposition}\label{ses corollary}
    Let $Z$ be as in Proposition \ref{tech prop 1}. Then, with respect previous notation, we have the following short exact sequence
    \[0\rightarrow H^1_Z(\Q_{S_{\Q}\cup Z}/\Q, M)\xrightarrow{\alpha_Z} H^1_{Y}(\Q_{\Sigma}/\Q, M)\xrightarrow{\pi_V} \bigoplus_{\ell\in V} \frac{ H^1(\Q_\ell, M) }{ H^1_{\op{nr}}(\Q_\ell, M)}\rightarrow 0.\]
    Here $\alpha_Z$ is the inflation map and $\pi_V$ is induced by the direct sum of restriction maps for the primes in $V$.
\end{proposition}
\begin{proof}
 The inflation map 
    \[\alpha_Z:H^1_Z(\Q_{S_{\Q}\cup Z}/\Q, M)\rightarrow H^1_{Y}(\Q_{\Sigma}/\Q, M)\] is injective and the sequence
    \[0\rightarrow H^1_Z(\Q_{S_{\Q}\cup Z}/\Q, M)\rightarrow H^1_{Y}(\Q_{\Sigma}/\Q, M)\rightarrow \bigoplus_{\ell\in V} \frac{ H^1(\Q_\ell, M) }{ H^1_{\op{nr}}(\Q_\ell, M)}\] is exact. In order to prove that the above is a short exact sequence, we must observe that 
    \[\# \left(H^1_Z(\Q_{S_{\Q}\cup Z}/\Q, M)\right)\prod_{\ell\in V}\# \left(\frac{ H^1(\Q_\ell, M) }{ H^1_{\op{nr}}(\Q_\ell, M)}\right)=\# \left(H^1_{Y}(\Q_{\Sigma}/\Q, M)\right).\]
Recall that for $\ell\in \mathfrak{T}_E$, we have 
    \begin{equation*}\begin{split}
       & \# H^1(\Q_\ell, M)=\# \op{Hom}\left(\op{G}_{\Q_\ell}, \Z/p^n \Z\right)=p^{2n};\\
       & \# H^1_{\op{nr}}(\Q_\ell, M)=\# \op{Hom}\left(\op{G}_{\Q_\ell}/\op{I}_\ell, \Z/p^n \Z\right)=p^{n}.\\
    \end{split}\end{equation*}
The conclusion now follows from the above equation and the assertion of Lemma \ref{prop 4.4}.
\end{proof}

Let $\mathfrak{B}_V$ be as in Definition \ref{B-V definition} and $\mathfrak{A}_V$ be the subset of $ \bigoplus_{\ell\in V} \frac{ H^1(\Q_\ell, M) }{ H^1_{\op{nr}}(\Q_\ell, M)}$ consisting of tuples $(z_\ell+H^1_{\op{nr}}(\Q_\ell, M))_{\ell\in V}$ where the class ${z_\ell}_{|\op{I}_\ell}:\op{I}_\ell\rightarrow M$ is a surjective homomorphism for all $\ell\in V$. Then, it is easy to see that 
\begin{equation}\label{size of A_T}\# \mathfrak{A}_V=\left((p-1) p^{n-1}\right)^{|V|}.\end{equation}
Let $H^1_{Y, V-\op{ram}}(\Q_\Sigma/\Q, M)$ be set of all cohomology classes $f\in H^1_Y(\Q_\Sigma/\Q, M)$ such that $\pi_V(f)\in \mathfrak{A}_V$. These are the classes $f\in H^1_Y(\Q_\Sigma/\Q, M)$ that are totally ramified at all primes in $V$. Note that this is only a subset of $H^1_Y(\Q_\Sigma/\Q, M)$, and not a subgroup. However, if $f\in H^1_Y(\Q_\Sigma/\Q, M)$ and $h \in \op{inf}_K\left(H^1(B, M)\right)$, then 
\[f\in H^1_{Y, V-\op{ram}}(\Q_\Sigma/\Q, M)\Leftrightarrow f+h\in H^1_{Y, V-\op{ram}}(\Q_\Sigma/\Q, M).\] Therefore, we may define an equivalence relation on $H^1_{Y, V-\op{ram}}(\Q_\Sigma/\Q, M)$ by $f\sim g$ if $f-cg\in \op{inf}_K\left(H^1(B, M)\right)$ for some $c \in {(\Z / p^n \Z)}^{\times}$, and set
\[\overline{H^1}_{Y, V-\op{ram}}(\Q_\Sigma/\Q, M):=H^1_{Y, V-\op{ram}}(\Q_\Sigma/\Q, M)/\sim\] to be the set of equivalence classes. 

\begin{proposition}\label{prop 4.6}
    With respect to the notation above, the following assertions hold.
    \begin{enumerate}
        \item[(i)]\label{p1 of prop 4.6} $\# \overline{H^1}_{Y, V-\op{ram}}(\Q_\Sigma/\Q, M)=\left(\frac{\# H^1_Z(\Q_{S_{\Q}\cup Z}/\Q, M)}{\# H^1(B, M)}\right)\times \bigg((p-1) p^{n-1}\bigg)^{|V|-1}$.
        \item[(ii)]\label{p2 of prop 4.6} The injection from Proposition \ref{bijection 1} restricts to a natural injection 
        \[\phi_V: \overline{H^1}_{Y, V-\op{ram}}(\Q_\Sigma/\Q, M)\hookrightarrow \mathfrak{B}_V,\] defined by $\phi_V(f):=K_f$. This map is well defined and independent of the equivalence class of $f$.
        \item[(iii)]\label{p3 of prop 4.6} The size of $\mathfrak{B}_V$ is given by 
        \[\# \mathfrak{B}_V \geq C \bigg((p-1) p^{n-1}\bigg)^{|V|},\] where $C:=\left(\frac{\# H^1_Z(\Q_{S_{\Q}\cup Z}/\Q, M)}{(p-1) p^{n-1}\# H^1(B, M)}\right)$ is a positive constant that does not depend on $V$. 
    \end{enumerate}
\end{proposition}
\begin{proof}
    From the definition above, it is clear that
    \[\# \overline{H^1}_{Y, V-\op{ram}}(\Q_\Sigma/\Q, M)=\frac{ \# H^1_{Y, V-\op{ram}}(\Q_\Sigma/\Q, M)}{(p-1)p^{n-1} \# \op{inf}_K\left(H^1(B, M)\right)}.\] Therefore, to compute $\# H^1_{Y, V-\op{ram}}(\Q_\Sigma/\Q, M)$, recall the short exact sequence
     \[0\rightarrow H^1_Z(\Q_{S_{\Q}\cup Z}/\Q, M)\rightarrow H^1_{Y}(\Q_{\Sigma}/\Q, M)\xrightarrow{\pi_V} \bigoplus_{\ell\in V} \frac{ H^1(\Q_\ell, M) }{ H^1_{\op{nr}}(\Q_\ell, M)}\rightarrow 0\] from Proposition \ref{ses corollary}. We identify 
     $$H^1_{Y, V-\op{ram}}(\Q_\Sigma/\Q, M) = \pi_V^{-1}\left(\mathfrak{A}_V\right).$$
     Note that if $f\in H^1_{Y, V-\op{ram}}(\Q_\Sigma/\Q, M)$ and $h\in H^1_Z(\Q_{S_{\Q}\cup Z}/\Q, M)$, then, $f+h\in H^1_{Y, V-\op{ram}}(\Q_\Sigma/\Q, M)$. Hence, we have 
     \[\frac{\# H^1_{Y, V-\op{ram}}(\Q_\Sigma/\Q, M)}{\# H^1_Z(\Q_{S_{\Q}\cup Z}/\Q, M)}=\# \mathfrak{A}_V.\]
     By equation \eqref{size of A_T}, we deduce that
     \[\# \overline{H}^1_{Y, V-\op{ram}}(\Q_\Sigma/\Q, M)=\left(\frac{\# H^1_Z(\Q_{S_{\Q}\cup Z}/\Q, M)}{\# H^1(B, M)}\right)\times \bigg((p-1) p^{n-1}\bigg)^{|V|}.\]
     This proves part(i). \\

     \par For part (ii), consider the map \[\tilde{\phi}_V: H^1_{Y, V-\op{ram}}(\Q_\Sigma/\Q, M)\rightarrow \mathfrak{B}_V\]
     defined by $\tilde{\phi}_V(f):=K_f$. We see that $K_f=K_{f'}$ if and only if $f-cf'$ lies in the image of $\op{inf}_K$ for some $c \in {(\Z/p^n\Z)}^{\times}$. Therefore, $\tilde{\phi}_V$ descends to an injection $\phi_V$. \\
     
 \par Part (iii) is an immediate consequence of parts (i) and (ii).
\end{proof}

\noindent We are now in a position to conclude the proof of our main result.
\begin{theorem}\label{main thm}
    Let $p\geq 5$ be a prime number, $K/\Q$ be an abelian number field and $E_{/\Q}$ be an elliptic curve. Setting $B:=\op{Gal}(K/\Q)$ and $T \simeq \Z/p^n \Z$, let $\G = B \ltimes T$. Assume that the conditions of Theorem \ref{main thm intro} are satisfied. Let $S_{\Q}$ denote the set of all rational primes $\ell$ such that
    \begin{itemize}
        \item $\ell$ is ramified in $K$, 
        \item $\ell$ is a prime at which $E$ has bad reduction,
    \end{itemize} 
    and let $S$ be the set of primes of $K$ that lie above $S_{\Q}$. Let $\mathfrak{T}_E$ be as in Definition \ref{definition of TE}. Fix a choice of a finite set of primes $Z \subset \mathfrak{T}_E$, determined by Lemma \ref{tech prop 1}. Let $V$ be any finite set of primes contained in $\mathfrak{T}_E$ that is disjoint from $S_{\Q}\cup Z$. Let $\mathfrak{B}_V$ be as in Definition \ref{B-V definition}. Then the following assertions hold:
\begin{enumerate}
    \item[(a)] $\# \mathfrak{B}_V \geq C \bigg((p-1) p^{n-1}\bigg)^{|V|}$, for a positive constant $C$,
    \item[(b)] for any $(\G, K)$-extension $L/K/\Q$ contained in $\mathfrak{B}_V$, $\op{Sel}_p(E/L)=0.$ In particular, the rank of $E(L)$ is $0$. 
\end{enumerate}
\end{theorem}
\begin{proof}
The first assertion is proven in part (iii) of Proposition \ref{prop 4.6}. For the second part, it suffices to show that the $(\G,K)$ extension $L$ satisfies the conditions of Proposition \ref{cor 2.7}. The only primes of $K$ that ramify in $L$ are those that lie above $Y= V \cup Z$. These primes are in $\mathfrak{T}_E$, and so by Lemma \ref{T_E lemma}, for all primes $v$ of $K$ that lie above them, we have that $\widetilde{E}(k_v)[p]=0$. This concludes the proof.
\end{proof}

\bigskip

\section{\bf A density result}\label{s 5}

\par Let $E$ be an elliptic curve over $\Q$, and $K/\Q$ be a Galois number field extension satisfying the conditions of Theorem \ref{main thm intro}. We now obtain a lower bound for the number of $(\mathcal{G}, K)$-extensions $L/K/\Q$ for which $\op{Sel}_p(E/L)=0$, ordered by absolute discriminant. Let $N_{\mathcal{G},K}(E;x)$ denote the total number of $(\mathcal{G}, K)$-extensions $L/K/\Q$ such that $\op{Sel}_p(E/L)=0$ and $|\Delta_L|\leq x$. We obtain an asymptotic lower bound for $N_{\mathcal{G},K}(E;x)$. Let $\mathfrak{T}_E$ be the set of primes as in Definition \ref{definition of TE}.

\begin{lemma}\label{lemma 4.8}
    Let $V$ be a finite set of primes in $\mathfrak{T}_E$ and let $L$ be a $(\mathcal{G}, K)$-extension in $\mathfrak{B}_V$. Then 
    \[|\Delta_L|=\left(\prod_{\ell\in V}\ell\right)^{|B|(p^n-1)}.\]
\end{lemma}
\begin{proof}
    By definition, each prime $\ell\in \mathfrak{T}_E$ is completely split in $K$. Since $L$ belongs to $\mathfrak{B}_V$, each of the primes $v\mid \ell$ of $K$ is totally ramified in $L$. Note that all primes that are ramified in $L$ are tamely ramified. Recall that 
    \[\Delta_L=\op{Norm}_{L/\Q} \left(\mathfrak{D}_{L/\Q}\right)\] where $\mathfrak{D}_{L/\Q}$ is the different ideal of $L/\Q$. It follows from well known properties of the different in a tamely ramified extension that \[|\Delta_L|=\left(\prod_{\ell\in V}\ell\right)^{|B|(p^n-1)},\] cf. \cite[Proposition 13]{Serrelocalfields}.
\end{proof}

\begin{theorem}\label{discriminant density theorem}
Let $E$ be an elliptic curve over $\Q$, and $K/\Q$ be a Galois number field extension satisfying the conditions of Theorem \ref{main thm intro}. With respect to earlier notation, we have
\[N_{\mathcal{G},K}(E;x)\gg x^{\frac{1}{|B|(p^n-1)}} \, \, (\log x)^{\left(p^{n-1}(p-1) \alpha\right)-1},\] where 
\[\alpha:=\left(\frac{p^2-p-1}{[K:\Q(\mu_p)](p-1)(p^2-1)}\right).\]
\end{theorem}

\begin{proof}
   The argument of proof similar to \cite[Theorem 2.4, p.5 line -3 to p.6 line -10.]{serredivisibilite}. We provide a sketch here for completeness. Setting $y:=x^{\frac{1}{|B|(p^n-1)}}$, denote by $M(y)$ the collection of all finite subsets $V$ of $\mathfrak{T}_E$ such that 
\[\prod_{\ell\in V}\ell\leq y.\]From Theorem \ref{main thm} and Lemma \ref{lemma 4.8}, we have
\begin{equation}\label{equation 4.6}
    N_{\G, K}(E;x)\geq \sum_{V\in M(y)} \# \mathfrak{B}_V.
\end{equation}
Thus, from Theorem \ref{main thm} part (a) and \eqref{equation 4.6} we get 
\begin{equation}\label{equation 4.7}
    N_{\G, K}(E;x)\geq C \sum_{V\in M(y)} r^{|V|},
\end{equation}
where $r:=p^{n-1}(p-1)$. Consider the function 
\[F(s):=\sum_{\substack{V\subset \mathfrak{T}_E \\ V \text{ finite}}} \frac{r^{|V|} }{\prod\limits_{\ell \in V} \ell^s}= \prod_{\ell\in \mathfrak{T}_E}\left(1+ \frac{r}{ \ell^{s}}\right).\] 

\noindent Thus, we see that
\[\log F(s)=r \, \sum_{\ell\in \mathfrak{T}_E} \frac{1}{\ell^s}+\theta_1(s),\] where $\theta_1(s)$ is holomorphic in $\op{Re}(s)\geq 1$; and as a consequence,
\[\log F(s) = r \alpha \, \log\left(\frac{1}{s-1}\right)+\theta_2(s), \] where $\alpha = \mathfrak{d}(\mathfrak{T}_E)$, which is computed in Lemma \ref{density lemma}, and $\theta_2(s)$ is holomorphic in $\op{Re}(s)\geq 1$. Thus, we deduce that 
\[F(s)=(s-1)^{- \,r  \alpha} \, h(s),\] where $h(s)$ is a non-zero holomorphic function in $\op{Re}(s)\geq 1$. By Delange's Tauberian theorem (cf. \cite[Theorem 7.28]{Tenenbaum}), we deduce that 
\[N_{\G, K}(E;x)\gg y  \, (\log y)^{r \alpha - 1} \gg x^{\frac{1}{|B|(p^n-1)}} \, (\log x)^{ \left(p^{n-1}(p-1) \alpha \right) - 1}.\] This proves the result. 
\end{proof}

\section*{Acknowledgments}
We thank both the referees for helpful comments on an earlier version of this paper. 

\bibliographystyle{alpha}
\bibliography{references}
\end{document}